\documentclass[11pt,a4paper]{article}
\usepackage{amsmath}
\usepackage{amsthm}
\usepackage{amsfonts, amssymb, amscd, latexsym}
\usepackage{eucal, enumerate, latexsym}
\usepackage{graphics}
\usepackage{graphicx}
\usepackage[hidelinks]{hyperref}
\usepackage[T1]{fontenc}
\usepackage{calc}
\newcommand{\R}{\mathbb R}

\newtheorem{theorem}{Theorem}[section]
\newtheorem{corollary}[theorem]{Corollary}
\newtheorem{proposition}[theorem]{Proposition}
\newtheorem{lemma}[theorem]{Lemma}
\newtheorem{definition}[theorem]{Definition}
\newtheorem{remark}[theorem]{Remark}

\newcommand{\orho}{\overline{\rho}}

\newcommand{\ds}{\displaystyle}

\renewcommand{\epsilon}{\varepsilon}

\renewcommand{\phi}{\varphi}

\newcommand{\C}[1]{\mathbf{C^{#1}}}
\newcommand{\Cc}[1]{\mathbf{C_c^{#1}}}
\newcommand{\Lloc}[1]{\mathbf{L^{#1}_{loc}}}

\newcommand{\ins}{\mathsf{I_0^c}}
\newcommand{\is}{\mathsf{I_0}}
\newcommand{\EL}{f_\ell^{-1}}
\newcommand{\ER}{f_r^{-1}}

\makeatletter

\newlength{\captionwidth}
\long\def\@makecaption#1#2{%
 \vskip 10\p@
 \setbox\@tempboxa\hbox{#1: #2}%
 \ifdim \wd\@tempboxa > \captionwidth 
  \hbox to\hsize{\hfil
  \parbox[t]{\captionwidth}{
  \small#1: \small#2\par}
  \hfil}
  \else
  \hbox to\hsize{\hfil\box\@tempboxa\hfil}%
 \fi}

\makeatother

\usepackage[usenames, dvipsnames]{color}

\numberwithin{equation}{section}

\let\originalleft\left
\let\originalright\right
\renewcommand{\left}{\mathopen{}\mathclose\bgroup\originalleft}
\renewcommand{\right}{\aftergroup\egroup\originalright}

\usepackage{psfrag}

\begin{document}

\allowdisplaybreaks

\title{Traveling waves for degenerate diffusive equations on networks}

\author{Andrea Corli\\
\small\textit{Department of Mathematics and Computer Science}\\
\small\textit{University of Ferrara, I-44121 Italy}
\\
\and Lorenzo di Ruvo\\
\small\textit{Department of Sciences and Methods for Engineering}\\
\small\textit{University of Modena and Reggio Emilia, I-42122 Italy}
\\
\and Luisa Malaguti\\
\small\textit{Department of Sciences and Methods for Engineering}\\
\small\textit{University of Modena and Reggio Emilia, I-42122 Italy} 
\\
\and Massimiliano D.~Rosini\\
\small\textit{Department of Mathematics}\\
\small\textit{Maria Curie-Sk{\l}odowska-University, PL-20031 Poland}}


\maketitle
\begin{abstract}
In this paper we consider a scalar parabolic equation on a star graph; the model is quite general but what we have in mind is the description of traffic flows at a crossroad. In particular, we do not necessarily require the continuity of the unknown function at the node of the graph and, moreover, the diffusivity can be degenerate. Our main result concerns a necessary and sufficient algebraic condition for the existence of traveling waves in the graph. We also study in great detail some examples corresponding to quadratic and logarithmic flux functions, for different diffusivities, to which our results apply.
 
\vspace{1cm}

\end{abstract}

\section{Introduction}
Partial differential equations on networks have been considered in the last years by several authors, in particular in the parabolic case; we quote for instance \cite{Dager-Zuazua, Garavello-Han-Piccoli_book, Garavello-Piccoli_book, Lagnese-Leugering-Schmidt, Pokornyi-Borovskikh, Below-Thesis}. According to the modeling in consideration and to the type of equations on the edges of the underlying graph, different conditions at the nodes are imposed. In most of the cases, precise results of existence of solutions are given, even for rather complicated networks.

In this paper, the main example we have in mind comes from traffic modeling, where the network is constituted by a crossroad connecting $m$ incoming roads with $n$ outgoing roads; the traffic in each road is modeled by the scalar diffusive equation
\begin{align}\label{e:system}
&\rho_{h,t} + f_h(\rho_h)_x = \left(D_h(\rho_h)\rho_{h,x}\right)_x,&
&h = 1,\ldots,m+n,
\end{align}
where $t$ denotes time and $x$ the position along the road.
In this case $\rho_h$ is a vehicle density; about the diffusivity $D_h(\rho_h)\ge0$ we do not exclude that it may vanish at some points. System \eqref{e:system} is completed by a condition of flux conservation at the crossroad, which implies the conservation of the total number of cars. Such a model is derived from the famous Lighthill-Whitham-Richards equation \cite{Lighthill-Whitham, Richards}. We refer to \cite{BTTV, Helbing2001, Lighthill-Whitham, Nelson_2000, Payne, SchonhofHelbing2007} for several motivations about the introduction of (possibly degenerate) diffusion in traffic flows and in the close field of crowds dynamics. We also refer to the recent books \cite{Garavello-Han-Piccoli_book, Garavello-Piccoli_book, Rosinibook} for more information on the related hyperbolic modeling.

We focus on a special class of solutions to \eqref{e:system}, namely, {\em traveling waves}. In the case of a single road, traveling waves are considered, for instance, in \cite{Nelson_2002}; in the case of a second-order model without diffusion but including a relaxation term, we refer to \cite{Flynn-Kasimov-Nave-Rosales-Seibold, Seibold-Flynn-Kasimov-Rosales}; for a possibly degenerate diffusion function and in presence of a source term, detailed results are given in \cite{Corli-diRuvo-Malaguti, Corli-Malaguti}. In the case of a network, the papers dealing with this subject, to the best of our knowledge, are limited to \cite{Below-Thesis, Below1994} for the semilinear diffusive case and to  \cite{Mugnolo-Rault} for the case of a dispersive equation. In these papers, as in most modeling of diffusive or dispersive partial differential equations on networks, both the continuity of the unknown functions and the Kirchhoff condition (or variants of it) are imposed at the nodes. We emphasize that while the {\em classical} Kirchhoff condition implies the conservation of the flow and then that of the mass, some variants of this condition are dissipative and, then, imply none of the conservations above. While these assumptions are natural when dealing with heat or fluid flows, they are much less justified in the case of traffic modeling, where the density must be allowed to jump at the node while the conservation of the mass must always hold. Moreover, they impose rather strong conditions on the existence of the profiles, which often amount to proportionality assumptions on the parameters in play.

In this paper we only require the conservation of the (parabolic) flux at the node, as in \cite{Coclite-Garavello}; differently from that paper and the other ones quoted above, we do {\em not} impose the continuity condition. A strong motivation for dropping this condition comes from the hyperbolic modeling \cite{ambroso2009coupling, Garavello-Han-Piccoli_book, Garavello-Piccoli_book, Rosinibook}; nevertheless, we show how our results simplify when such a condition is required. In particular, in Sections \ref{s:Greenshields} and \ref{s:log-fl} we provide explicit conditions for traveling wave solutions which do not satisfy the continuity condition; in some other cases, such a condition is indeed always satisfied. Our main results are essentially of algebraic nature and concern conditions about the end states, flux functions, diffusivities and other parameters which give rise to a traveling wave moving in the network. 

Here follows a plan of the paper. In Section \ref{s:m} we introduce the model and give some basic definitions; for simplicity we only focus on the case of a star graph. Section \ref{sec:TW1} deals with a general existence result in the case of a single equation; its proof is provided in Appendix \ref{s:A}. Section \ref{s:TWN} contains our main theoretical results about traveling waves in a network. In that section we characterize both  stationary/non-stationary and degenerate/non-degenerate waves; in particular, Theorem \ref{thm:1} contains an important necessary and sufficient condition that we exploit in the following sections. Section \ref{s:CC} focus on the continuity condition; in this case the conditions for the existence of traveling wave solutions are much stricter than in the previous case. Detailed applications of these results are provided in Sections \ref{s:Greenshields} for quadratic fluxes and in Section \ref{s:log-fl} for logarithmic fluxes; in particular, in subsection \ref{s:fGdL} and in the whole Section \ref{s:log-fl} the diffusivity is as in \cite{BTTV}. For simplicity, we only deal there with the case of a single ingoing road but we consider both constant and degenerate diffusivities.


\section{The model}\label{s:m}
In terms of graph theory, we consider a semi-infinite star-graph with $m$ incoming and $n$ outgoing edges; this means that the incidence vector $d\in\R^{m+n}$ has components $d_i=1$ for $i\in\mathsf{I} \doteq \{ 1, \ldots, m \}$ and $d_j=-1$ for $j\in\mathsf{J} \doteq \{ m+1, \ldots, m+n \}$. We also denote $\mathsf{H} \doteq \{ 1,\ldots, m+n \}$ and refer to Figure \ref{fig:net}. For simplicity, having in mind the example in the Introduction, we always refer to the graph as the network, to the node as the crossroad and to the edges as the roads. Then, incoming roads are parametrized by $x \in \R_-\doteq(-\infty,0]$ and numbered by the index $i$, outgoing roads by $x \in \R_+\doteq[0,\infty)$ and $j$; the crossroad is located at $x=0$ for both parameterizations. We denote the generic road by $\Omega_h$ for $h \in \mathsf{H}$; then $\Omega_i \doteq \R_-$ for $i \in \mathsf{I}$ and $\Omega_j \doteq \R_+$ for $j \in \mathsf{J}$. The network is defined as $\mathcal{N}\doteq \prod_{h\in\mathsf{H}}\Omega_h$.


\begin{figure}[htbp]\centering
\begin{psfrags}
      \psfrag{1}[l,B]{$\Omega_1$}
      \psfrag{2}[l,B]{$\Omega_2$}
      \psfrag{3}[l,B]{$\stackrel{\vdots}{\Omega}_i$}
      \psfrag{4}[l,B]{$\stackrel{\vdots}{\Omega}_m$}
      \psfrag{5}[l,B]{$\Omega_{m+1}$}
      \psfrag{6}[l,B]{$\stackrel{\vdots}{\Omega}_{m+2}$}
      \psfrag{7}[l,B]{$\stackrel{\vdots}{\Omega}_j$}
      \psfrag{8}[l,B]{$\Omega_{m+n}$}
\includegraphics[width=.3\textwidth]{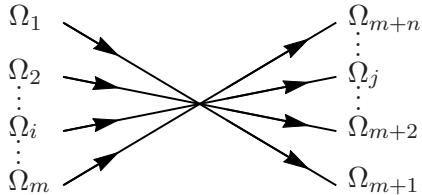}
\end{psfrags}
\caption{A network.}
\label{fig:net}
\end{figure}


Following the above analogy, we understand the unknown functions $\rho_h$ as vehicular densities in the roads $\Omega_h$, $h\in\mathsf{H}$; $\rho_h$ ranges in $[0,\orho_h]$, where $\orho_h$ is the maximal density in the road $\Omega_h$. Without loss of generality we assume that $\orho_h=1$ for every $h\in\mathsf{H}$; the general case is easily recovered by a change of variables and modifying \eqref{eq:Queen}-\eqref{e:1rhojrhoi} below for a multiplicative constant. With a slight abuse of notation we denote $\rho \doteq (\rho_1, \ldots, \rho_{m+n}): \R \times \mathcal{N} \to [0,1]^{m+n}$ understanding that $\rho(t,x_1,\ldots, x_{m+n})=(\rho_1(t,x_1),\ldots,\rho_{m+n}(t,x_{m+n}))$.

For each road we assign the functions $f_h$, the hyperbolic flux, and $D_h$, the diffusivity; we assume for every $h \in \mathsf{H}$
\begin{itemize}

\item[(f)] $f_h\in\mathbf{C}^1([0,1];\R_+)$ is strictly concave with $f_h(0)=f_h(1)=0$;

\item[(D)] $D_h\in\mathbf{C}^1([0,1];\R_+)$ and $D_h(\rho)>0$ for any $\rho \in (0,1)$.

\end{itemize}
We emphasize that in (D) we can possibly have either $D_h(0)=0$ or $D_h(1)=0$, or even both possibilities at the same time. The evolution of the flow is described by the equations
\begin{align}\label{eq:model}
&\rho_{h,t} + f_h(\rho_h)_x = \left(D_h(\rho_h)\rho_{h,x}\right)_x,&
&(t,x) \in \mathbb{R} \times \Omega_h,~ h \in \mathsf{H}.
\end{align}
Assumption (f) is standard when dealing with traffic flows \cite{BCGHP}. More precisely, in that case $f_h(\rho_h)=\rho_h \, v_h(\rho_h)$, where $v_h$ is the velocity.
Then, assumption (f) is satisfied if, for instance, $v_h\in\mathbf{C}^2([0,1];\R_+)$ is either linear or strictly concave, decreasing and satisfying $v_h(1) = 0$, see \cite{Lighthill-Whitham, Richards}.
The prototype of such a velocity satisfying (f) is $v_h(\rho) = V_h (1-\rho)$ with $V_h > 0$, which was introduced in \cite{Greenshields}; another example is given in \cite{pip}.
The simplest model for the diffusivity is then $D_h(\rho_h) = -\delta_h\rho_h \, v_h'(\rho_h)$, where $\delta_h$ is an anticipation length \cite{BTTV, Nelson_2002}.

The coupling among the differential equations in \eqref{eq:model} occurs by means of suitable conditions at the crossroad. In this paper, having in mind the previous example, we impose a condition on the {\em conservation of the total flow} at the crossroad, see \cite{Coclite-Garavello, Coclite-Garavello-Piccoli}; in turn, this implies the conservation of the mass. More precisely, we define the {\em parabolic flux} by
\[
F_h(\rho_h,\rho_{h,x}) \doteq f_h(\rho_h) - D_h(\rho_h) \, \rho_{h,x}
\]
and require
\begin{align}\label{eq:Queen}
&F_j\left(\rho_j(t,0^+), \rho_{j,x}(t,0^+)\right)
= \sum_{i \in \mathsf{I}} \alpha_{i,j} \, F_i\left(\rho_i(t,0^-), \rho_{i,x}(t,0^-)\right)
&\hbox{ for a.e.\ }t\in\R,~j \in \mathsf{J},
\end{align}
for given constant coefficients $\alpha_{i,j} \in (0, 1]$ satisfying
\begin{align}
\sum_{j \in \mathsf{J}} \alpha_{i,j} = 1,&
&i \in \mathsf{I}.
\label{e:1rhojrhoi}
\end{align}
Conditions \eqref{eq:Queen} and \eqref{e:1rhojrhoi} imply
\begin{align}\label{e:mass-cons}
&\sum_{j \in \mathsf{J}} F_j\left(\rho_j(t,0^+), \rho_{j,x}(t,0^+)\right)
=
\sum_{i \in \mathsf{I}} F_i\left(\rho_i(t,0^-), \rho_{i,x}(t,0^-)\right)
&\hbox{ for a.e.\ }t\in\R,
\end{align}
which is the conservation of the total flow at the crossroad. Conditions \eqref{eq:Queen} and \eqref{e:1rhojrhoi} deserve some comments. First, by no means they imply
\begin{align}\label{e:dens_cont}
&\rho_i(t,0^-) = \rho_j(t,0^+),&
&t \in \R, ~ (i,j)  \in \mathsf{I} \times \mathsf{J}.
\end{align}
Condition \eqref{e:dens_cont} is largely used, together with some Kirchhoff conditions, when dealing with parabolic equations in networks and takes the name of {\em continuity condition}. Second, above we assumed $\alpha_{i,j}>0$ for every $i$ and $j$. The case when $\alpha_{i,j}=0$ for some $i$ and $j$ would take into account the possibility that some outgoing $j$ roads are not allowed to vehicles coming from some incoming $i$ roads; this could be the case, for instance, if only trucks are allowed in road $i$ but only cars are allowed in road $j$. For simplicity, we do not consider this possibility. Third, we notice that assumption \eqref{eq:Queen} destroys the symmetry of condition \eqref{e:mass-cons}; indeed, with reference to the example of traffic flow, the loss of symmetry is due to the fact that all velocities $v_h$ are positive.

Then, we are faced with the system of equations \eqref{eq:model} that are coupled through \eqref{eq:Queen}, with the $\alpha_{i,j}$ satisfying \eqref{e:1rhojrhoi}. Solutions to \eqref{eq:model}-\eqref{eq:Queen} are meant in the weak sense, namely $\rho_h\in\mathbf{C^1}(\R\times\Omega_h;[0,1])$ a.e.; see also \cite{BCGHP, Garavello-Piccoli_book} for an analogous definition in the hyperbolic case. We do not impose any initial condition because we only consider traveling waves, which are introduced in the next sections.


\section{Traveling waves for a single equation}\label{sec:TW1}

In this section we briefly remind some definitions and results about traveling waves \cite{GK} for the {\em single} equation
\begin{align}\label{e:one}
&\rho_{h,t} + f_h(\rho_h)_x = \left(D_h(\rho_h)\rho_{h,x}\right)_x,&
&(t,x) \in \mathbb{R} \times \Omega_h,
\end{align}
where we keep for future reference the index $h$. Equation \eqref{e:one} has no source terms (differently from \cite{Below-Thesis, Below1994}) and then any constant is a solution; for simplicity we discard constant solutions in the following analysis.

\begin{definition}\label{def:solTWh} A weak solution $\rho_h(t,x)$  to \eqref{e:one} is a {\em traveling-wave solution} of \eqref{e:one} if $\rho_h(t,x) = \phi_h(x-c_h t)$ for $(t,x) \in \mathbb{R} \times \Omega_h$, for a non-constant profile $\phi_h:\R\to[0,1]$ and speed $c_h \in \R$.
\end{definition}

This definition coincides with that given in \cite{Mugnolo-Rault, Below_1988} because we are considering non-constant profiles. The profile must satisfy the equation
\begin{equation}\label{e:TheMarsVolta}
\left[F_h(\phi_h,\phi_h')- c_h \phi_h \right]'=0,
\end{equation}
namely,
\begin{align}\label{eq:phi}
&\left(D_h(\phi_h)\phi_h'\right)' - g_h'(\phi_h) \, \phi_h'=0,
\end{align}
in the weak sense, where
\begin{equation}\label{e:g}
g_h(\rho) \doteq f_h(\rho) - c_h \, \rho
\end{equation}
is the reduced flux, see \figurename~\ref{fig:gh}.

\begin{figure}[htbp]\centering
\begin{psfrags}
      \psfrag{a}[r,B]{$g_h(\ell_h^\pm)$}
      \psfrag{b}[c,B]{$\ell_h^-$}
      \psfrag{c}[c,B]{$\ell_h^+$}
      \psfrag{1}[c,B]{$1$}
      \psfrag{r}[c,B]{$\rho$}
      \psfrag{h}[l,c]{$y=c_h\,\rho$}
      \psfrag{y}[r,c]{$y$}
\includegraphics[width=.3\textwidth]{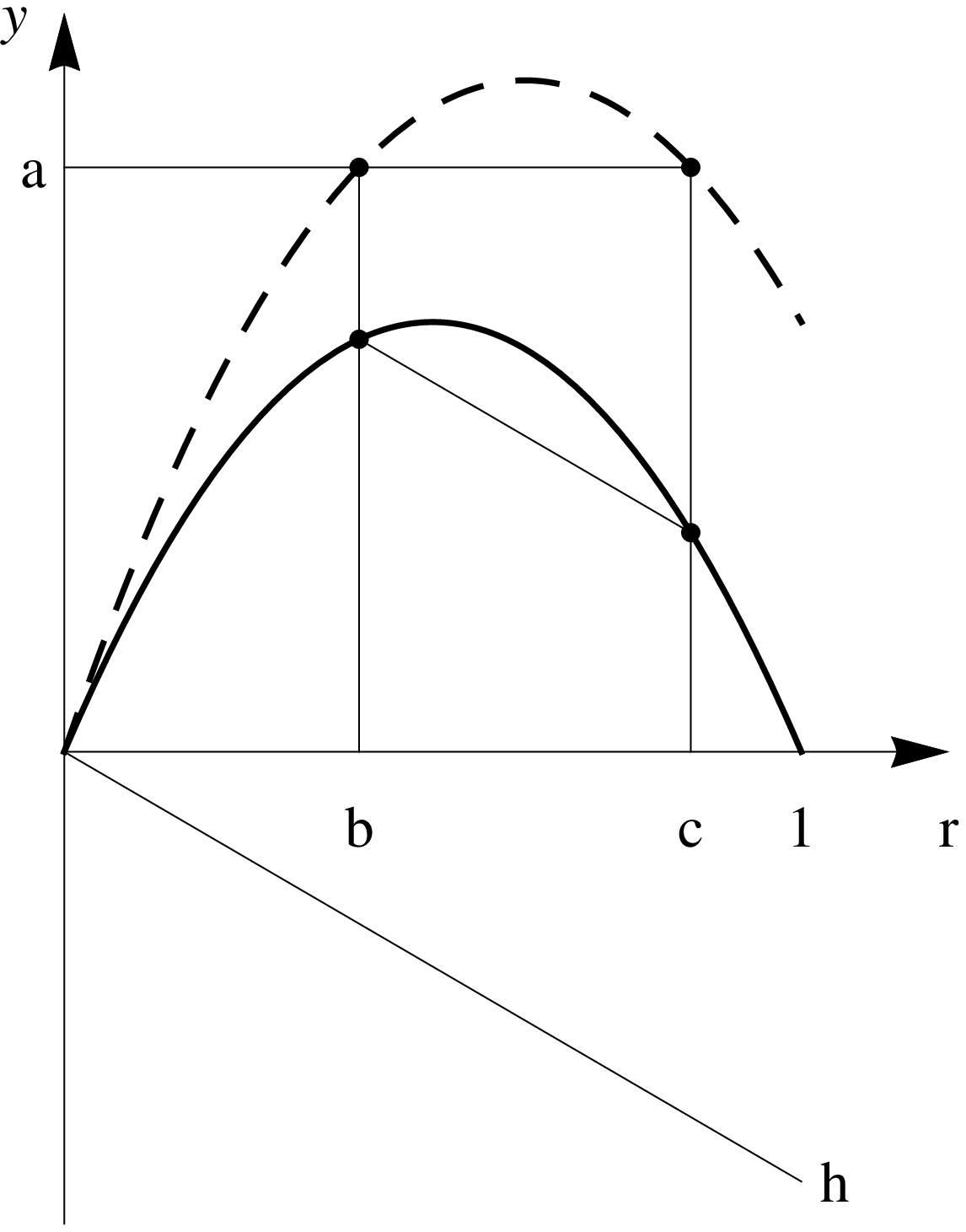}\hspace{.1\textwidth}
\includegraphics[width=.3\textwidth]{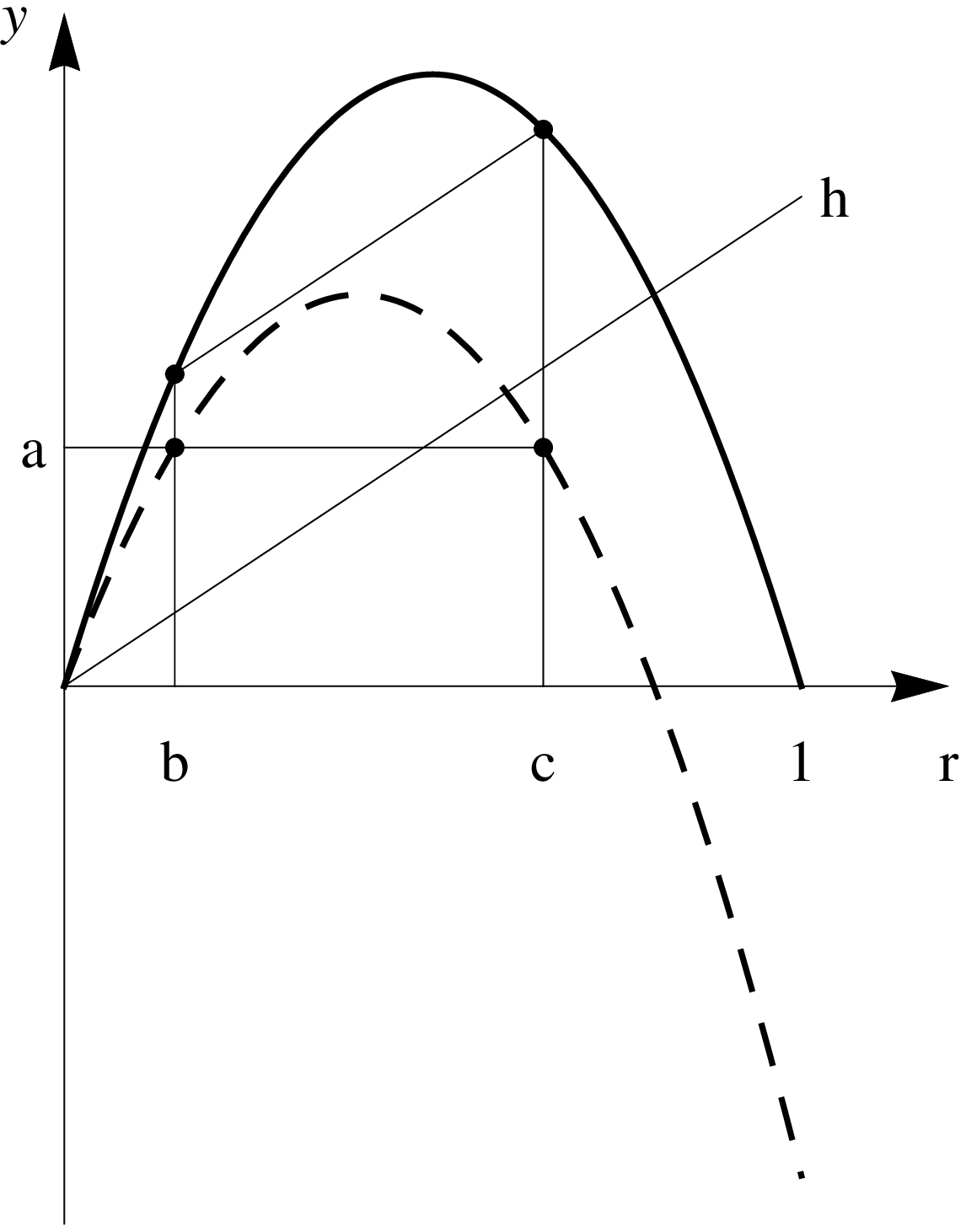}
\end{psfrags}
\caption{
{A flux $f_h$ satisfying (f), solid curve, and the corresponding reduced flux $g_h$ defined in \eqref{e:g}, dashed curve, in the case $c_h<0$, left, and in the case $c_h>0$, right.}}\label{fig:gh}
\end{figure}

This means that $\phi_h \in \C0(\mathbb{R}; [0,1])$, $D_h(\phi_h) \, \phi_h'\in \Lloc1(\mathbb{R};\mathbb{R})$ and
\begin{equation*}
\int_{\mathbb{R}} \left[D_h\left(\phi_h(\xi)\right)\phi_h'(\xi) - g_h\left(\phi_h(\xi)\right) \right] \psi'(\xi) \, {\rm d}\xi =0,
\end{equation*}
for every $\psi\in \Cc\infty(\mathbb{R};\mathbb{R})$. Equation \eqref{eq:phi} is  coupled with the limit conditions
\begin{equation}\label{e:infty}
\phi_h(\pm\infty) = \ell_h^\pm,
\end{equation}
for $\ell_h^\pm\in[0,1]$. Clearly, solutions to \eqref{eq:phi}-\eqref{e:infty} are determined up to a shift. We define
\begin{equation}\label{e:Ih}
I_h \doteq \left\{\xi \in \mathbb{R}\, : \, \ell_h^-<\phi_h(\xi)<\ell_h^+\right\}.
\end{equation}

The existence of profiles is a well-established result \cite{GK}; nevertheless, we state for completeness the following theorem, where we point out the qualitative properties of these fronts. The proof is deferred to Appendix \ref{s:A}.

\begin{theorem}\label{t:E}
Assume {\rm (f)} and {\rm (D)}. Equation \eqref{e:one} admits a traveling-wave solution $\rho_h$ with profile $\phi_h$ satisfying \eqref{e:infty} if and only if
\begin{align}\label{eq:ch}
&0\le\ell_h^- < \ell_h^+\le1&&\hbox{ and }&&c_h = \dfrac{f_h(\ell_h^+) - f_h(\ell_h^-)}{\ell_h^+ - \ell_h^-}.
\end{align}
We have that $\phi_h \in \mathbf{C}^2\left(I_h; (\ell_h^-, \ell_h^+)\right)$ is unique (up to shifts) and $\phi_h'(\xi)>0$ for $\xi \in I_h$; moreover, the following holds true.

\begin{itemize}
\item[{(i)}] $D_h(0)=0=\ell_h^-$ if and only if there exists $\nu_h^- \in \mathbb{R}$ such that $I_h \subseteq (\nu_h^-,\infty)$ and $\phi_h(\xi)=0$ for $\xi \le \nu_h^-$.
In this case
\begin{align}
\label{e:slope0}
&\lim_{\xi \downarrow \nu_h^-}\phi_h'(\xi)=\begin{cases}
\frac{\ell_h^+f_h'(0)-f_h(\ell_h^+)}{\ell_h^+D_h'(0)}& \hbox{if } D_h'(0)>0,\\
\infty & \hbox{if } D_h'(0)=0,
\end{cases}
\\
\label{eq:muh}
&\lim_{\xi \downarrow \nu_h^-}D_h\left(\phi_h(\xi)\right)\phi_h'(\xi) = 0.
\end{align}

\item[{(ii)}] $D_h(1)=0=1-\ell_h^+$ if and only if there exists $\nu_h^+ \in \mathbb{R}$ such that $I_h \subseteq (-\infty,\nu_h^+)$ and $\phi_h(\xi)=1$ for $\xi \ge \nu_h^+$.
In this case
\begin{align}\label{e:slope1}
&\lim_{\xi \uparrow \nu_h^+}\phi_h'(\xi)=\begin{cases}
\frac{\left(1-\ell_h^-\right)f_h'(1)+f_h(\ell_h^-)}
{\left(1-\ell_h^-\right) D_h'(1)}& \hbox{if } D_h'(1)<0,\\
\infty & \hbox{if } D_h'(1)=0,
\end{cases}
\\
\label{eq:nuh}
&\lim_{\xi \uparrow \nu_h^+}D_h\left(\phi_h(\xi)\right)\phi_h'(\xi) = 0.
\end{align}

\item[{(iii)}] In all the other cases $I_h=\mathbb{R}$  and
\begin{align}\label{eq:TosinAbasi}
\lim_{\xi \to \pm\infty}\phi_h'(\xi) = 0.
\end{align}
\end{itemize}
\end{theorem}

We observe that for $c_h$ given by \eqref{eq:ch}, we deduce by (f) that $g_h(\rho) \ge 0$ for all $\rho \in[\ell_h^-,\ell^+_h]$, see \figurename~\ref{fig:gh}. Moreover,  we have
\begin{equation}\label{eq:gh}
g_h(\ell_h^+) = g_h(\ell_h^-) =
-\frac{f_h(\ell_h^+) \, \ell_h^- - f_h(\ell_h^-) \, \ell_h^+}{\ell_h^+ - \ell_h^-}
\end{equation}
and no $\rho\ne\ell_h^\pm$ makes $g_h(\rho)$ equal to that value.

Theorem \ref{t:E} motivates the following definition.

\begin{definition}\label{def:stadeg}
A traveling-wave solution $\rho_h$ is {\em stationary} if $c_h=0$.
It is {\em degenerate} if at least one of conditions (i) or (ii) of Theorem \ref{t:E} holds.
\end{definition}

\begin{remark}\label{r:noC1}
A consequence of assumption {\em (f)} is that if $\rho_h$ is degenerate, then the profile $\phi_h$ is singular either at $\nu_h^-$ in case {(i)} or at $\nu_h^+$ in case {(ii)}, in the sense that $\phi_h'$ cannot be extended to the whole of $\R$ as a continuous function.
\end{remark}

In case {\em (i)} (or {\em (ii)}) of Theorem \ref{t:E} does not hold we define $\nu_h^- \doteq -\infty$ (respectively, $\nu_h^+ \doteq \infty$). In this way the interval $(\nu_h^-,\nu_h^+)$ is always defined and coincides with the interval $I_h$ defined in \eqref{e:Ih}:
\[
I_h = (\nu_h^-,\nu_h^+).
\]
The interval $I_h$ is bounded if and only if both {\em (i)} and {\em (ii)} hold; in this case $\rho_h$ is both degenerate and stationary.  As a consequence, if $\rho_h$ is non-stationary then $I_h$ is unbounded and coincides either with a half line (if $\rho_h$ is degenerate) or with $\R$ (if $\rho_h$ is non-degenerate).
At last, $\rho_h$ is degenerate if and only if either $\nu_h^-$ or $\nu_h^+$ is finite.

In the case of non-stationary traveling-wave solutions $\rho_h$ we use the notation
\begin{equation}\label{e:omegah}
\omega_h \doteq \min\{c_h^{-1}\nu_h^-,c_h^{-1}\nu_h^+\}.
\end{equation}

\begin{lemma}\label{lem:simple}
Let $\rho_h$ be a traveling-wave solution of \eqref{e:one}; then we have the following.
\begin{enumerate}
\item[(a)]
If $\rho_h$ is stationary, then it is degenerate if and only if $D_h(0)D_h(1)=0$ and $\ell_h^- = 0$ (hence $\ell_h^+=1$).
\item[(b)]
If $\rho_h$ is non-stationary, then it is degenerate if and only if one of the following equivalent statements hold:
\begin{itemize}
\item
either $D_h(0) = 0 = \ell_h^-$ or $D_h(1) = 0 = 1-\ell_h^+$, but not both;
\item
$\omega_h$ is finite.
\end{itemize}
In this case the function $\xi \mapsto \phi_h'(c_h\xi)$ is singular at $\xi = \omega_h$ and $\C1$ elsewhere.
\end{enumerate}
\end{lemma}
\begin{proof}
We recall that $\rho_h$ is degenerate if and only if either $D_h(0) = 0 = \ell_h^-$ or $D_h(1) = 0 = 1-\ell_h^+$. This means that at least one of the end states must be $0$ or $1$, say $0$; but then $c_h=0$ if and only if the other end state is $1$. This proves {\em (a)} and the first part of {\em (b)}.

Now, we prove the second part of \emph{(b)}.
Since $c_h \ne 0$, exactly one between {\em (i)} and {\em (ii)} of Theorem \ref{t:E} occurs, namely, exactly one between $\nu_h^-$ and $\nu_h^+$ is finite.
If $\nu_h^-$ is finite and $\nu_h^+=\infty$, then $c_h=f_h(\ell_h^+)/\ell_h^+>0$ and $\omega_h = c_h^{-1}\nu_h^-$ is finite.
By Remark \ref{r:noC1}, we know that $\xi\mapsto\phi_h'(\xi)$ is singular at $\xi=\nu_h^-$ and $\C1$ elsewhere, whence the regularity of $\xi \mapsto \phi_h'(c_h\xi)$.
Analogously, if $\nu_h^+$ is finite and $\nu_h^-=-\infty$, then $c_h=-f_h(\ell_h^-)/(1-\ell_h^-)<0$ and $\omega_h = c_h^{-1}\nu_h^+$ is finite. The statement about the smoothness of $\xi \mapsto \phi_h'(c_h\xi)$ is proved as above.

Finally, the converse is straightforward. In fact, if $\omega_h$ is finite, then either $\omega_h = c_h^{-1}\nu_h^-$ and $\nu_h^-$ is finite, or $\omega_h = c_h^{-1}\nu_h^+$ and  $\nu_h^+$ is finite; in both cases $\rho_h$ is degenerate.
\end{proof}

Because of the smoothness properties of the profile proved in Theorem \ref{t:E}, we can integrate equation \eqref{e:TheMarsVolta} in $(\xi_-,\xi)\subset I_h$ and we obtain
\[
c_h \phi_h(\xi) - F_h\left(\phi_h(\xi), \phi_h'(\xi)\right)
=
c_h \phi_h(\xi_-) - F_h\left(\phi_h(\xi_-), \phi_h'(\xi_-)\right).
\]
If $\xi_-\downarrow\nu_h^-$ in the previous expression, by applying \eqref{eq:muh} or \eqref{eq:TosinAbasi} we deduce
\begin{align}\label{e:Fhpre}
&F_h\left(\phi_h(\xi), \phi_h'(\xi)\right) = c_h \phi_h(\xi) + g_h(\ell_h^\pm),&
&\xi \in I_h.
\end{align}
We observe that \eqref{e:Fhpre} is trivially satisfied in case \emph{(i)} when  $\xi < \nu_h^-$ and in case \emph{(ii)} when  $\xi > \nu_h^+$; moreover, by a continuity argument, we deduce from \eqref{eq:muh} and \eqref{eq:nuh} that \eqref{e:Fhpre} is satisfied in case \emph{(i)} at $\xi = \nu_h^-$ and in case \emph{(ii)} at $\xi = \nu_h^+$, respectively. In conclusion, we have that \eqref{e:Fhpre} holds in the whole $\R$, namely
\begin{align}
\label{eq:Lorenzo4}
&D_h\left(\phi_h(\xi)\right) \, \phi_h'(\xi) = g_h\left(\phi_h(\xi)\right) - g_h(\ell_h^\pm),&
&\xi\in\R.
\end{align}


\section{Traveling waves in a network}\label{s:TWN}

In this section we consider the traveling-wave solutions of problem \eqref{eq:model}-\eqref{eq:Queen} in the network $\mathcal{N}$.
We first introduce the definition of traveling-wave solution in $\mathcal{N}$.
\begin{definition}\label{d:tw}
For any $h \in \mathsf{H}$, let $\rho_h$ be a traveling-wave solution of $\eqref{eq:model}_h$ in the sense of Definition \ref{def:solTWh} and set $\rho \doteq (\rho_1, \ldots, \rho_{m+n})$.
With reference to  Definition \ref{def:stadeg}, we say that:
\begin{itemize}
\item
$\rho$ is \emph{stationary} if each component $\rho_h$ is stationary;
\item
$\rho$ is {\em completely non-stationary} if none of its components is stationary;
\item
$\rho$ is \emph{degenerate} if at least one component $\rho_h$ is degenerate;
\item
$\rho$ is {\em completely degenerate} if each of its components is degenerate.
\end{itemize}
Finally, we say that $\rho$ is a \emph{traveling-wave solution} of problem \eqref{eq:model}-\eqref{eq:Queen} in the network $\mathcal{N}$ if \eqref{eq:Queen} holds.
\end{definition}

For brevity, from now on we simply write \lq\lq traveling wave\rq\rq\  for \lq\lq traveling-wave solution\rq\rq. In analogy to the notation above, we say that $\phi \doteq (\phi_1, \ldots, \phi_{m+n})$ is a profile for $\rho$ if $\phi_h$ is a profile corresponding to $\rho_h$ for every $h \in \mathsf{H}$.

For clarity of exposition, we collect our general results for stationary and non-stationary traveling waves in the following subsections.

\subsection{General results}

In this subsection, as well as in the following ones, we always assume (f) and (D) without explicitly mentioning it. Moreover, by Definition~\ref{d:tw} and Theorem \ref{t:E}, the end states and the speeds of the profiles must satisfy \eqref{eq:ch} for every $h\in\mathsf{H}$; both conditions in \eqref{eq:ch} are tacitly assumed as well.

\begin{proposition}\label{p:main}
The function $\phi$ is the profile of a traveling wave if and only if $\phi_h$ is a solution to \eqref{e:infty}-\eqref{eq:Lorenzo4} for any $h \in \mathsf{H}$ and
\begin{align}\label{eq:Tool}
&c_j \, \phi_j(c_j t) + g_j(\ell_j^\pm)
=
\sum_{i \in \mathsf{I}} \alpha_{i,j} \left[ c_i \, \phi_i(c_i t) + g_i(\ell_i^\pm) \right],&
&t \in \R, ~ j \in \mathsf{J}.
\end{align}
In \eqref{eq:Tool} any combination of the signs $\pm$ is allowed.\end{proposition}
\begin{proof}
By plugging $\rho_h(t,x)=\phi_h(x-c_ht)$ in \eqref{eq:Queen} and recalling that by Theorem~\ref{t:E} the profiles are continuous in $\R$, we obtain
\begin{align*}
&F_j\bigl(\phi_j(- c_j t), \phi_j'(- c_j t)\bigr)
=
\sum_{i \in \mathsf{I}} \alpha_{i,j} \, F_i\bigl(\phi_i(- c_i t), \phi_i'(- c_i t)\bigr),&
&t\in\R,\ j\in\mathsf{J},
\end{align*}
which is equivalent to \eqref{eq:Tool} by \eqref{eq:Lorenzo4}.
At last, we can clearly choose any combination of signs in \eqref{eq:Tool} because of \eqref{eq:gh}.
\end{proof}

Differently from what specified in Proposition \ref{p:main}, in the following the choice of the signs ``$\pm$'' follows the usual rules, i.e., top with top and bottom with bottom.

\begin{lemma}\label{l:minmax}
Assume that problem \eqref{eq:model}-\eqref{eq:Queen} admits a traveling wave. Then for any $j\in\mathsf{J}$ we have
\begin{align}\label{eq:maxfj}
\max\left\{f_j(\ell_j^-), f_j(\ell_j^+)\right \} &= \sum_{i \in \mathsf{I}} \alpha_{i,j} \, \max\left\{f_i(\ell_i^-), f_i(\ell_i^+)\right \},
\\\label{eq:minfj}
\min\left\{f_j(\ell_j^-), f_j(\ell_j^+)\right \} &= \sum_{i \in \mathsf{I}} \alpha_{i,j} \, \min\left\{f_i(\ell_i^-), f_i(\ell_i^+)\right \}.
\end{align}
\end{lemma}
\begin{proof}
Fix $j\in\mathsf{J}$.
We notice that \eqref{eq:Tool} is equivalent to
\begin{align*}
&\Upsilon_j(t) =
\sum_{i \in \mathsf{I}} \alpha_{i,j} \Upsilon_i(t),&t \in \R, ~ j \in \mathsf{J},
\end{align*}
where the map $t \mapsto \Upsilon_h(t) \doteq c_h \, \phi_h(c_ht) + g_h(\ell_h^-)$ is non-decreasing because the profiles are so, by Theorem \ref{t:E}. Since we can write $\Upsilon_h(t) = f_h(\ell_h^-) + c_h [\phi_h(c_ht) - \ell_h^-]$, we see that $\Upsilon_h$ ranges between $f_h(\ell_h^-)$ and $f_h(\ell_h^+)$ because of \eqref{eq:ch} and the fact that $\xi\mapsto\phi_h(\xi)$ takes values in $[\ell_h^-,\ell_h^+]$.  As a consequence,
\begin{align*}
&\lim_{t\to\infty}\Upsilon_h(t) = \max\left\{f_h(\ell_h^-),f_h(\ell_h^+)
\right\},&
&\lim_{t\to-\infty}\Upsilon_h(t) = \min\left\{f_h(\ell_h^-),f_h(\ell_h^+)
\right\}.
\end{align*}
Hence, by passing to the limit for $t\to\pm\infty$ in \eqref{eq:Tool} we obtain \eqref{eq:maxfj} and \eqref{eq:minfj}, respectively.
\end{proof}

\begin{lemma}\label{l:cj=0}
Assume that problem \eqref{eq:model}-\eqref{eq:Queen} admits a traveling wave. The traveling wave is stationary if and only if one of the following equivalent statements hold:
\begin{enumerate}
\item[(i)]
there exists ${\rm j} \in \mathsf{J}$ such that $c_{\rm j} = 0$;
\item[(ii)]
$c_i = 0$ for all $i \in \mathsf{I}$;
\item[(iii)]
$c_j = 0$ for all $j \in \mathsf{J}$.
\end{enumerate}
\end{lemma}
\begin{proof}
By subtracting \eqref{eq:minfj} to \eqref{eq:maxfj} we obtain
\[
|f_j(\ell_j^+) - f_j(\ell_j^-)| = \sum_{i \in \mathsf{I}} \alpha_{i,j} |f_i(\ell_i^+) - f_i(\ell_i^-)|.
\]
Since $c_h=0$ if and only if $f_h(\ell_h^-) = f_h(\ell_h^+)$, from the above equation we immediately deduce that $(i)$, $(ii)$ and $(iii)$ are equivalent. By the equivalence of {\em(ii)} and {\em(iii)}, a traveling wave is stationary if and only if one of the statements above holds.
\end{proof}

Lemma \ref{l:cj=0} shows that either a traveling wave is stationary, and then $c_h=0$ for every $h\in\mathsf{H}$, or it is non-stationary, and then
\begin{align}\label{e:cicj}
\hbox{there exists ${\rm i}\in\mathsf{I}$ such that $c_{\rm i}\ne0$ and $c_j\ne0$ for every $j\in\mathsf{J}$.}
\end{align}
Of course, by Lemma \ref{l:cj=0}, $c_{\rm i}\ne0$ for some ${\rm i}\in\mathsf{I}$ if and only if $c_j\ne0$ for every $j\in\mathsf{J}$.

\begin{proposition}\label{rem:DTP}
Fix $\ell_i^\pm \in [0,1]$ with $\ell_i^- < \ell_i^+$, $i \in \mathsf{I}$.
Then for any $j\in\mathsf{J}$ there exist $\ell_j^\pm \in [0,1]$ with $\ell_j^- < \ell_j^+$ and satisfying \eqref{eq:maxfj}-\eqref{eq:minfj} if and only if
\begin{equation}\label{e:ineq}
\begin{cases}
\displaystyle\max_{[0,1]}f_j > \sum_{i \in \mathsf{I}} \alpha_{i,j} \, \max\left\{f_i(\ell_i^-), f_i(\ell_i^+)\right\}&\hbox{if }c_1=\ldots=c_m=0,
\\
\displaystyle\max_{[0,1]}f_j \ge \sum_{i \in \mathsf{I}} \alpha_{i,j} \, \max\left\{f_i(\ell_i^-), f_i(\ell_i^+)\right\}&\hbox{otherwise}.
\end{cases}
\end{equation}
In this case, the end states $\ell_j^\pm$ are uniquely determined if and only if $c_i=0$ for every $i \in \mathsf{I}$.
\end{proposition}
\begin{proof}
Assume that there exist $\ell_j^\pm \in [0,1]$, with $\ell_j^- < \ell_j^+$, which satisfy \eqref{eq:maxfj}-\eqref{eq:minfj}. Then clearly we have $\max_{[0,1]}f_j \ge \sum_{i \in \mathsf{I}} \alpha_{i,j} \, \max\{f_i(\ell_i^-), f_i(\ell_i^+)\}$. If $c_i=0$ for every $i\in\mathsf{I}$, then we have $c_j=0$ for every $j\in\mathsf{J}$ by Lemma \ref{l:cj=0}; the equality $\max_{[0,1]}f_j= f(\ell_j^-)=f(\ell_j^+)$ would imply $\ell_j^-=\ell_j^+$ because of (f), a contradiction, and then $\max_{[0,1]}f_j> f(\ell_j^-)=f(\ell_j^+)$. This proves \eqref{e:ineq}.

\begin{figure}[htbp]\centering
\begin{psfrags}
      \psfrag{f}[c,B]{$f_j$}
      \psfrag{r}[l,B]{$\rho$}
      \psfrag{1}[l,B]{$1$}
      \psfrag{2}[c,B]{$\ell_j^-$}
      \psfrag{3}[c,B]{$\ell_j^+$}
      \psfrag{b}[c,b]{$\max\{f_j(\ell_j^-), f_j(\ell_j^+)\}$}
      \psfrag{a}[c,b]{$\min\{f_j(\ell_j^-), f_j(\ell_j^+)\}$}
\includegraphics[width=.9\textwidth]{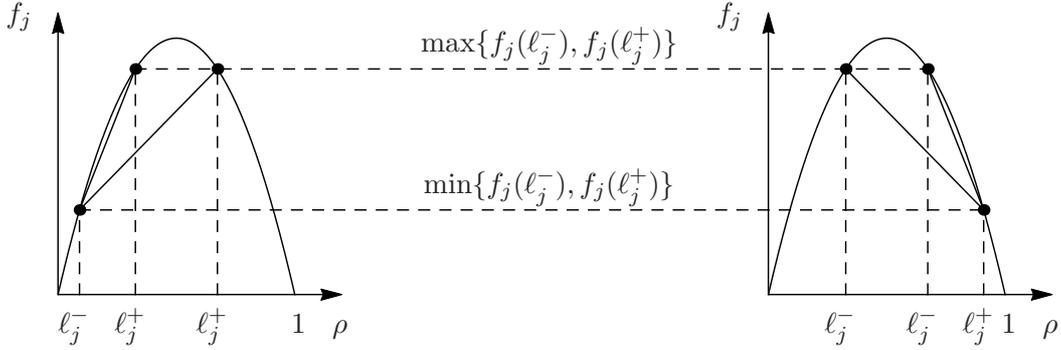}
\end{psfrags}
\caption{\label{fig:4cases}
{The values $\max\{f_j(\ell_j^-), f_j(\ell_j^+)\}$ and $\min\{f_j(\ell_j^-), f_j(\ell_j^+)\}$ equal the right-hand side of \eqref{eq:maxfj} and \eqref{eq:minfj}, respectively; the lines have slope $c_j\ne0$. Left: $c_j>0$. Right: $c_j<0$.}}
\end{figure}

Conversely, assume \eqref{e:ineq}. If $c_i=0$ for every $i\in\mathsf{I}$, then $\ell_j^-<\ell_j^+$ are uniquely determined because of the strict concavity of $f_j$. Assume, on the contrary, that $c_{\rm i}\ne0$ for some ${\rm i} \in \mathsf{I}$; then $c_j\ne0$ by Lemma \ref{l:cj=0}, i.e., $f_j(\ell_j^-)\ne f_j(\ell_j^+)$. Thus \eqref{eq:maxfj}-\eqref{eq:minfj} determine exactly four possible choices of end states $\ell_j^\pm$ with $\ell_j^- < \ell_j^+$, see Figure \ref{fig:4cases}.
\end{proof}

By Proposition \ref{rem:DTP} and Lemma \ref{l:cj=0} we deduce that the end states $\ell_j^\pm$ are uniquely determined in terms of the end states $\ell_i^\pm$ if and only if the traveling wave is stationary and the first condition in \eqref{e:ineq} holds.

We now give an algebraic result about determining the end states of the outgoing profiles in terms of the end states of the ingoing ones. We introduce
\begin{equation}\label{e:Lij}
L_{i,j}^\pm \doteq \begin{cases}
\ell_i^\pm & \hbox{ if } c_i \, c_j \ge 0,
\\
\ell_i^\mp & \hbox{ if } c_i \, c_j <0.
\end{cases}
\end{equation}

\begin{proposition}\label{p:Tool}
Assume that problem \eqref{eq:model}-\eqref{eq:Queen} admits a traveling wave. Then for any $j\in\mathsf{J}$ we have
\begin{align}\label{eq:fjpm}
f_j(\ell_j^\pm) &=
\sum_{i \in \mathsf{I}} \alpha_{i,j} \, f_i(L_{i,j}^\pm).
\end{align}
Moreover, \eqref{eq:fjpm} is equivalent to \eqref{eq:maxfj}-\eqref{eq:minfj}.
\end{proposition}
\begin{proof}
Fix $j\in\mathsf{J}$.
By Lemma \ref{l:minmax} it is sufficient to prove that \eqref{eq:fjpm} is equivalent to \eqref{eq:maxfj}-\eqref{eq:minfj}.
If $c_j>0$, and then $f_j(\ell_j^+) > f_j(\ell_j^-)$, by \eqref{e:Lij} we have
\begin{align*}
\max\left\{f_i(\ell_i^-), f_i(\ell_i^+)\right \} &=
\begin{cases}
f_i(\ell_i^+)&\hbox{if }c_i\ge0
\\
f_i(\ell_i^-)&\hbox{if }c_i<0
\end{cases}
=f_i(L_{i,j}^+),
\\
\min\left\{f_i(\ell_i^-), f_i(\ell_i^+)\right\} &=
\begin{cases}
f_i(\ell_i^-)&\hbox{if }c_i\ge0
\\
f_i(\ell_i^+)&\hbox{if }c_i<0
\end{cases}
=f_i(L_{i,j}^-),
\end{align*}
and therefore \eqref{eq:fjpm} is equivalent to \eqref{eq:maxfj}-\eqref{eq:minfj}.
The case $c_j<0$ is analogous.
If $c_j=0$, then $f_j(\ell_j^+) = f_j(\ell_j^-)$ and by Lemma \ref{l:cj=0} we have $f_i(\ell_i^+) = f_i(\ell_i^-)$ for any $i \in \mathsf{I}$. In this case formulas \eqref{eq:maxfj}-\eqref{eq:minfj} reduce to a single equation, which coincides with \eqref{eq:fjpm}.
\end{proof}

\subsection{The stationary case}

In this short subsection we briefly consider stationary traveling waves.

\begin{theorem}\label{t:1}
Problem \eqref{eq:model}-\eqref{eq:Queen} admits infinitely many stationary traveling waves; such waves are characterized by the conditions on the end states
\begin{align}\label{basic}
&f_h(\ell_h^+) = f_h(\ell_h^-),&
&f_j(\ell_j^-) = \sum_{i \in \mathsf{I}} \alpha_{i,j} \, f_i(\ell_i^-)&\hbox{ for }h\in\mathsf{H},\ j\in\mathsf{J}.
\end{align}
\end{theorem}
\begin{proof}
Clearly, \eqref{basic} is trivially satisfied if $\ell_h^-=0$ and $\ell_h^+ = 1$ for all $h \in \mathsf{H}$. We claim that there exist infinitely many choices of $\ell_1^\pm,\ldots,\ell_{m+n}^\pm$ satisfying \eqref{basic}.  To prove the claim, we choose $\ell_i^\pm \in [0,1]$, with $\ell_i^- < \ell_i^+$, such that $f_i(\ell_i^-) = f_i(\ell_i^+)$ are sufficiently small to satisfy the first condition in \eqref{e:ineq} for all $j \in \mathsf{J}$.
Then, by a continuity argument, we can choose $\ell_j^\pm \in [0,1]$ so that $\ell_j^- < \ell_j^+$ and $f_j(\ell_j^-) = f_j(\ell_j^+) = \sum_{i \in \mathsf{I}} \alpha_{i,j} \, f_i(\ell_i^-)$. This proves the claim.

With this choice of the end states, by Theorem~\ref{t:E} we deduce the existence of a stationary traveling wave in each road satisfying \eqref{eq:model}. At last we notice that, in the stationary case, condition \eqref{eq:Tool} is equivalent to the latter condition in \eqref{basic}.
\end{proof}

Clearly, if both $D_h(0)\ne0$ and $D_h(1)\ne0$ for every $h\in\mathsf{H}$, then problem \eqref{eq:model}-\eqref{eq:Queen} admits no degenerate traveling wave. However, even in the general case, the proof of Theorem \ref{t:1} shows that \eqref{eq:model}-\eqref{eq:Queen} admits infinitely many {\em non-degenerate} stationary traveling waves: just choose $0\ne\ell_h^-<\ell_h^+\ne1$ satisfying \eqref{basic}. Moreover, if there exists ${\rm h}\in\mathsf{H}$ such that either $D_{\rm h}(0)=0$ or $D_{\rm h}(1)=0$, then \eqref{eq:model}-\eqref{eq:Queen} admits also infinitely many {\em degenerate} stationary traveling waves: just choose $\ell_{\rm h}^-=0 = 1 - \ell_{\rm h}^+$ and determine the other end states by \eqref{basic}.

\subsection{The non-stationary case}

In this subsection we consider non-stationary traveling waves.
By Lemma \ref{l:cj=0} this is equivalent to consider the scenario in \eqref{e:cicj}: there exists ${\rm i}\in\mathsf{I}$ such that $f_{\rm i}(\ell_{\rm i}^-) \ne f_{\rm i}(\ell_{\rm i}^+)$ and $f_j(\ell_j^-) \ne f_j(\ell_j^+)$ for every $j\in\mathsf{J}$.
We can therefore introduce the following notation:
\begin{align}\label{eq:kappaj}
&c_{i,j} \doteq \frac{c_i}{c_j}, &
&A_{i,j} \doteq \alpha_{i,j} \, c_{i,j},
&k_j \doteq
\ds \sum_{i \in \ins} \left[ A_{i,j} \, L_{i,j}^\pm \right] - \ell_j^\pm,&
&\kappa_j
\doteq c_jk_j,&
\end{align}
where $L_{i,j}$ is defined in \eqref{e:Lij} and
\begin{align*}
&\is \doteq \{i\in\mathsf{I}:c_i=0\} = \{i\in\mathsf{I}:f_i(\ell_i^-) = f_i(\ell_i^+)\},&
&\ins \doteq \mathsf{I} \setminus \is.
\end{align*}
We notice that $\ins\ne\emptyset$ by \eqref{e:cicj} and that both $\is$ and $\ins$ depend on the end states $\ell_i^\pm$, $i\in\mathsf{I}$, indeed.
Moreover, $k_j$ is well defined because by \eqref{eq:fjpm}
\[
\sum_{i \in \ins} A_{i,j} \left[ L_{i,j}^+ - L_{i,j}^- \right] =
\sum_{i \in \ins} \alpha_{i,j} c_j^{-1} \left[ f_i(L_{i,j}^+) - f_i(L_{i,j}^-) \right] =
c_j^{-1} \left[ f_j(\ell_j^+) - f_i(\ell_j^-) \right] = \ell_j^+-\ell_j^-.
\]
Finally, by (f) we deduce that
\begin{equation*}
\hbox{for no $j \in \mathsf{J}$ we have both $\ell_j^-=0=1-\ell_j^+$.}
\end{equation*}

\begin{proposition}\label{p:expl-profiles}
The function $\phi$ is the profile of a non-stationary traveling wave if and only if $\phi_h$ is a solution to \eqref{e:infty}-\eqref{eq:Lorenzo4} for any $h \in \mathsf{H}$ and
\begin{align}\label{eq:1DevinTownsend}
\phi_j(\xi) & =
\sum_{i \in \ins} \left[ A_{i,j} \, \phi_i\left(c_{i,j} \, \xi\right) \right]
- k_j,&
&\xi\in\R,~j\in\mathsf{J}.
\end{align}
\end{proposition}
\begin{proof}
By Proposition \ref{p:main} it is sufficient to prove that by \eqref{e:cicj} condition \eqref{eq:Tool} is equivalent to \eqref{eq:1DevinTownsend}.
By \eqref{eq:gh} we have $g_i(\ell_i^+) = g_i(\ell_i^-) = g_i(L_{i,j}^+) = g_i(L_{i,j}^-)$ and then by \eqref{eq:fjpm} we have $\kappa_j = g_j(\ell_j^\pm) - \sum_{i \in \mathsf{I}} \alpha_{i,j} g_i(L_{i,j}^\pm)$.
Hence, by \eqref{e:cicj}, with the change of variable $\xi = c_jt$, condition \eqref{eq:Tool} is
\[
c_j \, \phi_j(\xi) =
- g_j(\ell_j^\pm) + \sum_{i \in \mathsf{I}} \alpha_{i,j} \, \left[ c_i \, \phi_i(c_{i,j}\xi) + g_i(L_{i,j}^\pm) \right]
=
\sum_{i \in \mathsf{I}} \alpha_{i,j} \, c_i \, \phi_i(c_{i,j}\xi) - \kappa_j,
\]
that is equivalent to \eqref{eq:1DevinTownsend}.
\end{proof}

We observe that $k_j$ and \eqref{eq:1DevinTownsend} can be written in a little bit more explicit form by avoiding the use of $L_{i,j}^\pm$ as follows
\begin{align}\label{eq:kappajbis}
k_j &= \sum_{i \in \ins} \left[ A_{i,j} \, \frac{\ell_i^- + \ell_i^+}{2}\right] - \frac{\ell_j^- + \ell_j^+}{2},
\\\nonumber
\phi_j(\xi)
&=
\frac{\ell_j^- + \ell_j^+}{2}
+
\sum_{i \in \ins} A_{i,j} \left[ \phi_i\left(c_{i,j} \, \xi\right) - \frac{\ell_i^- + \ell_i^+}{2}\right].
\end{align}
Proposition \ref{p:expl-profiles} shows how each outgoing profile $\phi_j$ can be expressed by \eqref{eq:1DevinTownsend} in terms of the ingoing profiles $\phi_i$, $i\in\mathsf{I}$.
We know a priori that $\phi_j$ is increasing and its end states are contained in the interval $[0,1]$.
Now, we prove a sort of converse implication, which shows that these properties of the profile $\phi_j$ are enjoined by the function defined by the right-hand side of \eqref{eq:1DevinTownsend}.

\begin{lemma}\label{l:phij}
Let $\phi_i$,  for $i\in\mathsf{I}$, be the profiles provided by Theorem~\ref{t:E} and assume that $\ins\ne\emptyset$; fix $j\in\mathsf{J}$ and consider any $l_j^\pm\in[0,1]$ satisfying \eqref{eq:fjpm} and such that, for the corresponding $c_j$, it holds $c_j\ne0$.
Then $l_j^- < l_j^+$. Moreover, denote by $\ell_j(\xi)$ the right-hand side of \eqref{eq:1DevinTownsend}; then $\xi\mapsto\ell_j(\xi)$ is non-decreasing and $\ell_j(\pm\infty)=l_j^\pm$.

\end{lemma}
\begin{proof}
Since by Theorem \ref{t:E} we know that $\ell_i^-<\ell_i^+$, then by \eqref{eq:fjpm}
\begin{align*}
l_j^+-l_j^- =
c_j^{-1} \left[ f_j(l_j^+) - f_j(l_j^-) \right] =
\sum_{i \in \mathsf{I}} \alpha_{i,j} \, c_j^{-1} \left[ f_i(L_{i,j}^+) - f_i(L_{i,j}^-) \right]
\\=
\sum_{i \in \mathsf{I}} \alpha_{i,j} \, c_{i,j} \left[L_{i,j}^+- L_{i,j}^-\right] =
\sum_{i \in \ins} \alpha_{i,j} \, |c_{i,j}| \left[\ell_i^+ - \ell_i^-\right] &> 0.
\end{align*}
By definition of $\ell_j$ we have
\(
\ell_j'(\xi) =
\sum_{i \in \ins} \alpha_{i,j} \, c_{i,j}^2 \, \phi_i'\left(c_{i,j} \, \xi\right)
\)
for a.e.\ $\xi\in\R$, hence $\xi\mapsto\ell_j(\xi)$ is non-decreasing since all profiles $\phi_i$ do.
Moreover, $\ell_j(\pm\infty) = l_j^\pm$ because by the definitions of $\ell_j$ and $\kappa_j$  we have
\[
c_j\,\ell_j(\pm\infty) =
\sum_{i \in \ins} \left[ \alpha_{i,j} \, c_i \, L_{i,j}^\pm \right]
- \kappa_j =
c_j \, l_j^\pm.\qedhere
\]
\end{proof}

We notice that Proposition \ref{p:expl-profiles} exploits condition \eqref{eq:Queen} through its expression \eqref{eq:Tool} for the profiles; the diffusivities $D_h$ are not involved in \eqref{eq:1DevinTownsend}. Indeed, Proposition \ref{p:expl-profiles} imposes strong necessary conditions on the diffusivities as we discuss now as a preparation to \eqref{eq:RATM}.

We notice that if both $\nu_h^-$ and $\nu_h^+$ are finite, then $\ell_h^-=0=1-\ell_h^+$ and consequently $c_h=0$; therefore either $\nu_h^-$ or $\nu_h^+$ (possibly both) is infinite for any $h\in \ins \cup \mathsf{J}$.

The following result is similar to Lemma \ref{l:cj=0}.
\begin{lemma}\label{l:prop-deg}
Problem \eqref{eq:model}-\eqref{eq:Queen} admits a degenerate non-stationary traveling wave $\rho$ if and only if at least one of the following conditions holds:
\begin{itemize}
\item[(A)] for some $i \in \is$ we have $D_i(0) D_i(1)= 0$ and $\ell_i^- = 0$ (hence $\ell_i^+=1$);
\item[(B)] for every $h \in \ins\cup\mathsf{J}$ we have either $D_h(0) = 0 = \ell_h^-$ or $D_h(1) = 0 = \ell_h^+-1$, but not both. In this case we have
\begin{align}\label{eq:minchIh}
&\omega_i = \omega_j\doteq \omega,&
&i\in\ins,\ j \in \mathsf{J}.
\end{align}
\end{itemize}
\end{lemma}
\begin{proof}
Let us introduce the following conditions:
\begin{enumerate}
\item[$(B)'$] for some $i \in \ins$ we have either $D_i(0) = 0 = \ell_i^-$ or $D_i(1) = 0 = \ell_i^+-1$, but \emph{not} both;
\item[$(B)''$] for some $j \in \mathsf{J}$ we have either $D_j(0) = 0 = \ell_j^-$ or $D_j(1) = 0 = \ell_j^+-1$, but \emph{not} both.
\end{enumerate}
Clearly $(B)$ implies both $(B)'$ and $(B)''$.
Moreover, by Lemma \ref{lem:simple} and \eqref{e:cicj}, problem \eqref{eq:model}-\eqref{eq:Queen} admits a degenerate non-stationary traveling wave $\rho$ if and only if at least one of the conditions \emph{(A)}, $(B)'$ and $(B)''$ holds.
To complete the proof it is therefore sufficient to show that \emph{(B)}, $(B)'$ and $(B)''$ are equivalent.
By Lemma \ref{lem:simple}{\em (b)} and \eqref{e:cicj}, the conditions \emph{(B)}, $(B)'$ and $(B)''$ are respectively equivalent to
\begin{enumerate}
\item[(I)] $\omega_h$ is finite for every $h \in \ins \cup \mathsf{J}$,
\item[(II)] for some $i\in\ins$ we have that $\omega_i$ is finite,
\item[(III)] for some $j\in\mathsf{J}$ we have that $\omega_j$ is finite,
\end{enumerate}
where $\omega_h$ is defined in \eqref{e:omegah}.
Differentiating \eqref{eq:1DevinTownsend} gives
\begin{align}
\phi_j'(c_j\xi) & =
\sum_{i \in \ins} \alpha_{i,j} c_{i,j}^2 \phi_i'\left(c_i\xi\right)&
&\hbox{ for a.e.\ }\xi\in\R,\ j \in \mathsf{J}.
\label{e:squares}
\end{align}
More precisely, by Lemma \ref{lem:simple}, formula \eqref{e:squares} holds for $\xi\in\R\setminus\left(\{\omega_j\}\cup\bigcup_{i\in\ins}\omega_i\right)$; moreover, by the same lemma we know that $\xi\mapsto\phi_h'(c_h\xi)$ is singular at $\xi = \omega_h$ and $\C1$ elsewhere, for $h\in\ins\cup\mathsf{J}$.
Hence, \eqref{e:squares} implies \eqref{eq:minchIh}. By \eqref{eq:minchIh} we have that the above statements (I), (II) and (III) are equivalent and then also \emph{(B)}, $(B)'$ and $(B)''$ are so.
\end{proof}

As for Lemma \ref{l:cj=0}, we notice that Lemma \ref{l:prop-deg} implies that a non-stationary traveling wave $\rho$ is either non-degenerate, and then $\rho_h$ is non-degenerate for every $h \in \mathsf{H}$, or $\rho$ is degenerate, and then either there exists ${\rm i} \in \is$ such that $\rho_{\rm i}$ is degenerate, or $\rho_h$ is degenerate for all $h \in \ins \cup \mathsf{J}$.
In both cases a non-stationary traveling wave $\rho$ satisfies \eqref{eq:minchIh}.

When modeling traffic flows it is natural to use different diffusivities, which however share some common properties.
For instance, this led to consider in \cite{BTTV, Corli-Malaguti} the following subcase of (D):
\begin{itemize}

\item[(D1)] $D_h$ satisfies (D) and $D_h(0)=0$, $D_h(1)>0$, for every $h\in\mathsf{H}$.

\end{itemize}

The proof of the following result is an immediate consequence of Lemma \ref{l:prop-deg} and, hence, omitted.

\begin{corollary}
Assume that problem \eqref{eq:model}-\eqref{eq:Queen} has a non-stationary traveling wave $\rho$ and  {\rm (D1)} holds. Then $\rho$ is degenerate if and only at least one of the following conditions holds:
\begin{itemize}
\item[(A)] for some $i \in \is$ we have $\ell_i^- = 0$ (hence $\ell_i^+ = 1$);
\item[(B)] for \emph{every} $h \in \ins\cup\mathsf{J}$ we have $\ell_h^- = 0$ (hence $\ell_h^+ \neq 1$).
\end{itemize}
\end{corollary}

The case when $D_h$ satisfies (D) and $D_h(0)=0=D_h(1)$ for every $h\in\mathsf{H}$, see \cite{Corli-diRuvo-Malaguti, Corli-Malaguti}, can be dealt analogously.

The next result is the most important of this paper; there, we give necessary and sufficient conditions for the existence of non-stationary traveling waves in a network. About its statement, let us recall Theorem \ref{t:E}: we have $\phi_h'(\xi)=0$ in case {\em (i)} if $\xi<\nu_h^-$ or in case {\em (ii)} if $\xi>\nu_h^+$. Since $\phi_h$ satisfies equation \eqref{eq:Lorenzo4}, we are led to extend the quotient $\ell\mapsto \frac{g_h(\ell)-g_h(\ell_h^-)}{D_h(\ell)}$ to the whole of $\R$ by defining
\begin{equation}\label{e:gammah}
\gamma_h(\ell) \doteq
\begin{cases}
\frac{g_h(\ell)-g_h(\ell_h^-)}{D_h(\ell)} & \hbox{ if }D_h(\ell)\ne0,
\\
0& \hbox{ if }D_h(\ell)=0.
\end{cases}
\end{equation}
In fact, when $\ell$ is replaced by $\phi_h(\xi)$, then $\gamma_h(\ell)=\phi_h'(\xi)$ for $\xi \in \R \setminus \{\nu_h^-, \nu_h^+\}$. We remark that condition $D_h(\ell)=0$ occurs at most when either $\ell=0$ or $\ell=1$. To avoid the introduction of the new notation \eqref{e:gammah}, in the following we simply keep on writing $\frac{g_h(\ell)-g_h(\ell_h^-)}{D_h(\ell)}$ for $\gamma_h(\ell)$. As a consequence, any non-stationary traveling wave of problem \eqref{eq:model}-\eqref{eq:Queen} satisfies
\begin{align}\label{e:agree}
&\phi_h'(\xi)
=
\dfrac{g_h\left(\phi_h(\xi)\right) - g_h(\ell_h^-)}{D_h\left(\phi_h(\xi)\right)},&
&\xi \in\R\setminus\{ \nu_h^-, \, \nu_h^+\}, ~h \in \mathsf{H}.
\end{align}

\begin{theorem}\label{thm:1}
Assume conditions {\rm (f)} and {\rm (D)}.
Problem \eqref{eq:model}-\eqref{eq:Queen} admits a non-stationary traveling wave if and only if the following condition holds.

\begin{itemize}
\item[($\mathcal{T}$)] There exist $\ell_1^\pm,\ldots,\ell_m^\pm\in [0,1]$ with $\ell_i^- < \ell_i^+$, $i\in\mathsf{I}$, such that:
\begin{enumerate}

\item[(i)]  $\ins\ne\emptyset$;

\item[(ii)] for any $j \in \mathsf{J}$ there exist $\ell_j^\pm\in[0,1]$ satisfying \eqref{eq:fjpm} and such that $f_j(\ell_j^-) \ne f_j(\ell_j^+)$;

\item[(iii)] for any $j \in \mathsf{J}$ we have
\begin{align}\label{eq:RATM}
&\dfrac{g_j\left(\ell_j(c_j \, \xi)\right) - g_j(\ell_j^-)}{D_j\left(\ell_j(c_j \, \xi)\right)}
=
\sum_{i \in \ins} A_{i,j} \, c_{i,j} \,
\dfrac{g_i\left(\phi_i( c_i \, \xi)\right) - g_i(\ell_i^-)}{D_i\left(\phi_i( c_i \, \xi)\right)}&
&\hbox{for a.e.\ }\xi \in \R,
\end{align}
where $\phi_1,\ldots,\phi_m$ are solutions to \eqref{e:infty}-\eqref{eq:Lorenzo4} and, for $k_j$ as in \eqref{eq:kappaj},
\begin{align}\label{e:elljn}
&\ell_j(\xi) \doteq \sum_{i \in \ins} \left[ A_{i,j} \, \phi_i\left(c_{i,j} \, \xi\right) \right]
- k_j,&
&\xi\in\R.
\end{align}
\end{enumerate}
\end{itemize}
\end{theorem}

\begin{proof}
First, assume that problem \eqref{eq:model}-\eqref{eq:Queen} admits a non-stationary traveling wave $\rho$ with profiles $\phi_h$, end states $\ell_h^\pm$ and speeds $c_h$, for $h\in\mathsf{H}$. By Theorem~\ref{t:E} we have that $\ell_h^\pm$ and $c_h$ satisfy \eqref{eq:ch}.
By Proposition~\ref{p:expl-profiles} the profiles $\phi_h$ satisfy \eqref{e:infty}-\eqref{eq:Lorenzo4} and \eqref{eq:1DevinTownsend}.
The end states $\ell_j^\pm$, $j\in\mathsf{J}$, satisfy \eqref{eq:fjpm} by Proposition~\ref{p:Tool}.
Since $\rho$ is non-stationary we are in the scenario given by \eqref{e:cicj}: $\ins\ne\emptyset$ and $f_j(\ell_j^-) \ne f_j(\ell_j^+)$ for all $j \in \mathsf{J}$.
By \eqref{e:agree} with $h = j$ we have
\begin{equation}\label{e:post}
\phi_j'(c_j\xi)
=
\dfrac{g_j\left(\phi_j(c_j\xi)\right) - g_j(\ell_j^-)}{D_j\left(\phi_j(c_j\xi)\right)}
\end{equation}
for $\xi \in \R$ in the non-degenerate case and for $\xi \in\R\setminus\{ \omega\}$ with $\omega$ given by \eqref{eq:minchIh} in the degenerate case.
On the other hand, by differentiating \eqref{eq:1DevinTownsend} and applying \eqref{e:agree} with $h = i$ we deduce
\begin{align}
\phi_j'(\xi) &=
\sum_{i \in \ins} A_{i,j} \, c_{i,j} \, \phi_i'( c_{i,j} \xi)
=
\sum_{i \in \ins} A_{i,j} \, c_{i,j} \,
\dfrac{g_i\left(\phi_i( c_{i,j} \xi)\right) - g_i(\ell_i^-)}{D_i\left(\phi_i( c_{i,j} \xi)\right)}
\label{e:g2}
\end{align}
for $\xi \in \R$ in the non-degenerate case and for $\xi \in\R\setminus\{ \omega\}$ with $\omega$ given by \eqref{eq:minchIh} in the degenerate case.
Identity \eqref{eq:RATM} follows because $\ell_j \equiv \phi_j$ by \eqref{eq:1DevinTownsend} and by comparing \eqref{e:post}, \eqref{e:g2}.

Conversely, assume that condition ($\mathcal{T}$) holds.
We remark that the existence of $\phi_i$, $i \in \mathsf{I}$, is assured by Theorem~\ref{t:E}.
Fix  $j\in\mathsf{J}$.
By defining $\phi_j \doteq \ell_j$ we obtain \eqref{eq:1DevinTownsend}.
We know by assumption that $\ins \ne \emptyset$, $\ell_j^\pm \in [0,1]$ satisfy \eqref{eq:fjpm} and $c_j \ne 0$; we can apply therefore Lemma~\ref{l:phij} and deduce that $\ell_j^-<\ell_j^+$, $\phi_j$ is non-decreasing and satisfies \eqref{e:infty} with $h=j$.
By Proposition \ref{p:expl-profiles}, what remains to prove is that $\phi_j$ satisfies \eqref{eq:Lorenzo4}.
But by \eqref{eq:1DevinTownsend} we deduce \eqref{e:g2} for a.e.\ $\xi \in \R$, because $\phi_1,\ldots,\phi_m$ satisfy \eqref{eq:Lorenzo4} and hence, recalling the extension \eqref{e:gammah}, also \eqref{e:agree}; then by \eqref{eq:RATM} we conclude that $\phi_j$ satisfies \eqref{e:agree} for a.e.\ $\xi \in \R$ and then  \eqref{eq:Lorenzo4} for a.e.\ $\xi \in \R$.
Finally, \eqref{eq:Lorenzo4} holds by the regularity ensured by Theorem \ref{t:E} for the profiles.
\end{proof}

\begin{remark}
Fix $\ell_i^\pm \in [0,1]$, $i \in \mathsf{I}$, so that $\ell_i^- < \ell_i^+$ and \eqref{e:ineq} holds. We know by Proposition~\ref{rem:DTP} that for every $j\in\mathsf{J}$ there exists $(\ell_j^-,\ell_j^+) \in [0,1]^2$, with $\ell_j^- < \ell_j^+$, that satisfies \eqref{eq:fjpm}, but it is not unique. If beside \eqref{eq:fjpm} we impose also \eqref{eq:RATM}, then we may have three possible scenarios: such $(\ell_j^-,\ell_j^+)$ either does not exist, or it exists and is unique, or else it exists but is not unique. We refer to Subsections~\ref{s:fGdC} and \ref{s:fGdL} for further discussion.
\end{remark}


\section{The continuity condition}\label{s:CC}

In this section we discuss the case when solutions to \eqref{eq:model}-\eqref{eq:Queen} are {\em also} required to satisfy the continuity condition \eqref{e:dens_cont}; this makes the analysis much easier because \eqref{e:dens_cont} implies several strong conditions.

First, we provide the main results about traveling waves satisfying condition \eqref{e:dens_cont}. We point out that some of the consequences below have already been pointed out in \cite{Mugnolo-Rault, Below-Thesis, Below1994} in the case that some Kirchhoff conditions replace the conservation of the total flow \eqref{eq:Queen}. In order to emphasize the consequences of the continuity condition \eqref{e:dens_cont}, the first two parts of the following lemma do {\em not} assume that also condition \eqref{eq:Queen} holds.
\begin{lemma}\label{l:+cc}
For any $h \in \mathsf{H}$, let $\rho_h$ be a traveling wave of $\eqref{eq:model}_h$ in the sense of Definition \ref{def:solTWh} and set $\rho \doteq (\rho_1, \ldots, \rho_{m+n})$; then the following holds for every $(i,j) \in \mathsf{I}\times\mathsf{J}$ and $h\in\mathsf{H}$.
\begin{enumerate}
\item[(i)] $\rho$ satisfies \eqref{e:dens_cont} if and only if
\begin{align}\label{e:cont_profiles}
&\phi_j(c_jt) = \phi_i(c_i t) \doteq \Phi(t),&
&t\in\R.
\end{align}
\item[(ii)] If $\rho$ satisfies \eqref{e:dens_cont}, then either it is stationary (hence \eqref{e:cont_profiles} reduces to $\phi_j(0) = \phi_i(0)$), or it is completely non-stationary and the speeds $c_h$ have the same sign (hence $c_{i,j}>0$).
In the latter case, $\rho$ is either non-degenerate or completely degenerate; moreover
\begin{align}
\label{eq:cI}
(c_j^{-1}I_j) &= (c_i^{-1}I_i) \doteq \mathcal{I},
\\
\label{eq:Beardfish}
\ell_j^\pm&=\ell_i^\pm=L_{i,j}^\pm \doteq \ell^\pm,
\\
\label{e:gjgi}
c_j\frac{g_j(\ell)-g_j(\ell^\pm)}{D_j(\ell)} &= c_i\frac{g_i(\ell)-g_i(\ell^\pm)}{D_i(\ell)},& &\ell\in(\ell^-,\ell^+).
\end{align}
\item[(iii)]
If $\rho$ is non-stationary and satisfies both \eqref{eq:Queen} and \eqref{e:dens_cont}, then
\begin{align}\label{e:4eqn}
&c_j =\sum_{i\in\mathsf{I}} \alpha_{i,j}c_i,&
&\sum_{j\in\mathsf{J}}c_j =\sum_{i\in\mathsf{I}}c_i,&
&\kappa_j=0,&
&\sum_{i \in \mathsf{I}} A_{i,j} = 1.
\end{align}
\end{enumerate}
\end{lemma}

\begin{proof}
We split the proof according to the items in the statement.

\smallskip
\noindent

\begin{enumerate}[{\em (i)}]

\item Condition \eqref{e:dens_cont} and \eqref{e:cont_profiles} are clearly equivalent.

\item Since we are discarding constant profiles, by \eqref{e:cont_profiles} we have that either $c_h=0$ for all $h\in\mathsf{H}$ or $c_h\ne0$ for all $h\in\mathsf{H}$. The stationary case is trivial; therefore we consider below only the non-stationary case and assume that $c_h\ne0$ for all $h\in\mathsf{H}$.
By differentiating \eqref{e:cont_profiles} with respect to $t$ we deduce
\begin{align}\label{e:cont_diff}
&c_j\phi_j'(c_jt) = c_i\phi_i'(c_i t)&&\hbox{ for a.e. }t\in\R.
\end{align}
Then \eqref{e:cont_diff} implies that either $\rho$ is non-degenerate or it is completely degenerate.
Moreover \eqref{e:cont_diff} implies \eqref{eq:cI} because, by Lemma \ref{lem:simple}, we have that $\rho_h$ is degenerate if and only if the map $\xi \mapsto \phi_h'(c_h \xi)$ is singular at $\xi = \omega_h \in \R$ and $\C1$ elsewhere.
By taking $t \in \mathcal{I}$ in \eqref{e:cont_diff} we deduce that $c_i$ and $c_j$ have the same sign. As a consequence we have $L_{i,j}^\pm=\ell_i^\pm$ and then $\ell_i^\pm = \ell_j^\pm$ by letting $t\to\pm\infty$ in \eqref{e:cont_profiles}.
By \eqref{e:cont_profiles}, \eqref{eq:Lorenzo4} and \eqref{eq:Beardfish} we have $D_h\left(\Phi(\xi)\right) \Phi'(\xi) = c_h \left[g_h(\Phi(\xi)) - g_h(\ell^\pm)\right]$ for all $h\in \mathsf{H}$, whence \eqref{e:gjgi} by the extension \eqref{e:gammah}.

\item To deduce $\eqref{e:4eqn}_1$, we differentiate \eqref{eq:Tool} and then exploit \eqref{e:cont_diff}. Formula $\eqref{e:4eqn}_2$ follows by summing $\eqref{e:4eqn}_1$ with respect to $j$ and by \eqref{e:1rhojrhoi}.
By $\eqref{e:4eqn}_1$ we have $\sum_{i \in \mathsf{I}} A_{i,j} = \sum_{i \in \mathsf{I}} \alpha_{i,j} c_i c_j^{-1}=1$, which proves $\eqref{e:4eqn}_4$.
Finally, \eqref{eq:Beardfish} and $\eqref{e:4eqn}_4$ imply $\eqref{e:4eqn}_3$.\qedhere
\end{enumerate}
\end{proof}

In the following proposition we deal with stationary traveling waves satisfying condition \eqref{e:dens_cont}.

\begin{proposition}\label{rem:SC}
Problem \eqref{eq:model}-\eqref{eq:Queen} admits infinitely many stationary traveling waves satisfying \eqref{e:dens_cont}; their end states $\ell_h^\pm$ satisfy \eqref{basic} and are such that $S \doteq \bigcap_{h\in\mathsf{H}}(\ell_h^-,\ell_h^+)\ne\emptyset$.
\end{proposition}
\begin{proof}
By \eqref{e:cont_profiles} condition \eqref{e:dens_cont} holds in the stationary case if and only if $\phi_i(0)=\phi_j(0)$ for $(i,j)\in\mathsf{I}\times\mathsf{J}$.
Recalling the proof of Theorem \ref{t:1}, it is sufficient to take $\ell_h^\pm \in [0,1]$ satisfying \eqref{basic} and such that $S \ne \emptyset$, $\ell_0 \in S$ and the unique solution $\phi_h$ to \eqref{e:infty}-\eqref{eq:Lorenzo4} such that $\phi_h(0) = \ell_0$. There are infinitely many of such profiles because of the arbitrariness of $\ell_h^\pm$.
\end{proof}

We point out that condition $S=\emptyset$ can occur if the functions $f_h$ assume their maximum values at different points. This is not the case when the following condition $\eqref{e:proportionality}_1$ is assumed.

The following result is analogous to Theorem \ref{thm:1} in the case \eqref{e:dens_cont} holds.

\begin{theorem}\label{thm:2}
Assume conditions {\rm (f)} and {\rm (D)}. Problem \eqref{eq:model}-\eqref{eq:Queen} admits a (completely) non-stationary  traveling wave satisfying \eqref{e:dens_cont} if and only if the following condition holds.
\begin{itemize}
\item[($\mathcal{T}_c$)]  There exist $\ell^\pm\in [0,1]$ with $\ell^- < \ell^+$, such that for any $h\in\mathsf{H}$, $i \in \mathsf{I}$ and $j \in \mathsf{J}$
\begin{align}\label{Tc1}
f_h(\ell^-) &\ne f_h(\ell^+),
\\
\label{eq:fjicontinuo}
f_j(\ell^\pm)&=\sum_{i \in \mathsf{I}}\alpha_{i,j}f_i(\ell^\pm),
\\
\label{Tc4}
c_j \dfrac{g_j\left(\phi_j(c_j t)\right) - g_j(\ell^-)}{D_j\left(\phi_j(c_j t)\right)}
&=
c_i \dfrac{g_i\left(\phi_i(c_i t)\right) - g_i(\ell^-)}{D_i\left(\phi_i(c_i t)\right)}&
&\hbox{for a.e.\ }t\in\R,
\end{align}
where $c_h$ is given by \eqref{eq:ch}, $\phi_h$ is a solution to \eqref{eq:Lorenzo4} such that $\varphi_h(\pm\infty) = \ell^\pm$ and $\varphi_1(0) = \ldots = \varphi_{m+n}(0)$.
\end{itemize}
\end{theorem}
\begin{proof}
Assume that condition ($\mathcal{T}_c$) holds; the other implication is obvious.
We remark that the existence of $\phi_1,\ldots,\phi_{m+n}$ is assured by Theorem \ref{t:E}; indeed, for any $\ell_0 \in (\ell^-,\ell^+)$, up to shifts it is always possible to assume that $\varphi_h(0) = \ell_0$, $h \in \mathsf{H}$.
By \eqref{eq:fjicontinuo} we have $\eqref{e:4eqn}_4$ because
\[
\sum_{i \in \mathsf{I}} A_{i,j} = \sum_{i \in \mathsf{I}} \alpha_{i,j} \frac{f_i(\ell^+) - f_i(\ell^-)}{f_j(\ell^+) - f_j(\ell^-)} = 1.
\]
By \eqref{Tc1} we have that $\ins = \mathsf{I}$, $\is = \emptyset$ and $\rho$ corresponding to the profile $\varphi \doteq (\varphi_1, \ldots, \varphi_{m+n})$ is completely non-stationary.
Then $\eqref{e:4eqn}_4$ and \eqref{eq:kappajbis} imply $\eqref{e:4eqn}_3$, namely $k_j = 0$.
By Lemma \ref{l:+cc} \emph{(i)} and Proposition \ref{p:expl-profiles} it remains to prove \eqref{e:cont_profiles} and \eqref{eq:1DevinTownsend}.
We start with \eqref{e:cont_profiles}.
Clearly \eqref{e:cont_profiles} holds for $t=0$ because $\phi_h(0) = \ell_0$, $h\in\mathsf{H}$.
Then by the extension \eqref{e:gammah} and \eqref{Tc4} we have
\begin{align*}
\frac{{\rm d}}{{\rm d} t} \left[ \phi_j(c_jt) - \phi_i(c_i t) \right]
=
c_j \dfrac{g_j\left(\phi_j(c_j t)\right) - g_j(\ell^-)}{D_j\left(\phi_j(c_j t)\right)}
-
c_i \dfrac{g_i\left(\phi_i(c_i t)\right) - g_i(\ell^-)}{D_i\left(\phi_i(c_i t)\right)}
= 0.
\end{align*}
Therefore we conclude that \eqref{e:cont_profiles} holds.
Finally, \eqref{eq:1DevinTownsend} follows immediately from \eqref{e:cont_profiles}, $\eqref{e:4eqn}_3$ and $\eqref{e:4eqn}_4$.
\end{proof}

Consider in particular the case when the functions $f$ and $D$ satisfy {\rm (f)} and {\rm (D)}, respectively, and assume that
\begin{align}\label{e:proportionality}
&f_h(\ell) \doteq v_h f(\ell),&
&D_h(\ell) \doteq \delta_h D(\ell),&
&\ell \in [0,1],
\end{align}
for some constants $v_h,\delta_h>0$. Denote
\begin{align}\label{eq:vdelta}
&v_{i,j} \doteq \frac{v_i}{v_j},&
&\delta_{i,j} \doteq \frac{\delta_i}{\delta_j}.
\end{align}
We notice that now we have
\begin{equation}\label{eq:v=c}
v_{i,j}=c_{i,j}.
\end{equation}
In the following proposition we apply Theorem \ref{thm:2} when \eqref{e:proportionality} is assumed; in this case conditions \eqref{eq:fjicontinuo} and \eqref{Tc4} no longer depend on the end states and the statement is somewhat simplified.

\begin{proposition}\label{p:super}
Assume \eqref{e:proportionality} with $f$ and $D$ satisfying {\rm (f)} and {\rm (D)}, respectively. Problem \eqref{eq:model}-\eqref{eq:Queen} admits a (completely) non-stationary traveling wave satisfying \eqref{e:dens_cont} if and only if for every $i\in\mathsf{I}$ and $j\in\mathsf{J}$ we have
\begin{align}\label{e:Mugnolo}
&v_{i,j}^2=\delta_{i,j} &\hbox{ and }&
&\sum_{i \in \mathsf{I}} \alpha_{i,j} v_{i,j} = 1.
\end{align}
\end{proposition}
\begin{proof}
We only need to translate condition $(\mathcal{T}_c)$ to the current case. Let $\ell^\pm \in [0,1]$ with $\ell^-<\ell^+$ and $f(\ell^-)\ne f(\ell^+)$.
By \eqref{e:proportionality} it is obvious that \eqref{Tc1} is satisfied.
If $\ell^-=0$ or $\ell^+=1$ condition \eqref{eq:fjicontinuo} is satisfied by (f). In all the other cases \eqref{eq:fjicontinuo} is equivalent to
\(
\sum_{i \in \mathsf{I}} \alpha_{i,j} v_{i,j} = 1
\)
by \eqref{e:proportionality}.
Similarly, condition \eqref{Tc4} reduces to $c_jv_j\delta_j^{-1}=c_iv_i\delta_i^{-1}$ and hence, by \eqref{eq:v=c}, it is equivalent to $v_{i,j}^2=\delta_{i,j}$.
\end{proof}

Remark that by \eqref{eq:v=c} condition $\eqref{e:Mugnolo}_2$ is equivalent to $\eqref{e:4eqn}_4$.


\section{Application to the case of a quadratic flux, \texorpdfstring{$m=1$}{}}\label{s:Greenshields}

In this section we assume \eqref{e:proportionality} for some constants $v_h,\delta_h>0$,  $D$ satisfying {\rm (D)} and the quadratic flux \cite{Greenshields}
\[
f(\rho) \doteq \rho \, (1-\rho),
\]
with no further mention. 
The case when only $\eqref{e:proportionality}_1$ holds is doable and follows with slight modifications. We use the notation introduced in \eqref{eq:vdelta}.

For simplicity, in the whole section we focus on the case $m=1$, see Figure \ref{fig:stargraphm1}, even without explicitly mentioning it. Then $\mathsf{I}=\{1\}$, $\mathsf{J}=\{2,\ldots,n+1\}$, $\mathsf{H}=\{1,2,\ldots, n+1\}$. The general case $m>1$ offers no further difficulties than heavier calculations. 


\begin{figure}[htbp]
\centering
\begin{psfrags}
      \psfrag{1}[l,B]{$\Omega_1$}
      \psfrag{5}[l,B]{$\Omega_2$}
      \psfrag{6}[l,B]{$\stackrel{\vdots}{\Omega}_3$}
      \psfrag{7}[l,B]{$\stackrel{\vdots}{\Omega}_j$}
      \psfrag{8}[l,B]{$\Omega_{n+1}$}
\includegraphics[width=.3\textwidth]{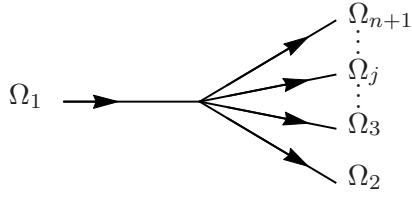}
\end{psfrags}
\caption{A network with $m=1$.}
\label{fig:stargraphm1}
\end{figure}


In this case, condition \eqref{eq:ch} becomes
\begin{align}\label{e:Gspeedg}
&0\le\ell_h^- < \ell_h^+\le1&
&\hbox{ and }&
&c_h = v_h \, [1-\ell_h^+ - \ell_h^-].
\end{align}
In particular, by $\eqref{e:Gspeedg}_2$
\begin{equation}\label{e:phinonst}
\hbox{ $\rho_h$ is stationary $\iff$ $\ell_h^++\ell_h^-=1$.}
\end{equation}
Moreover, $g_h(\ell) = v_h \, \ell \, [\ell_h^+ + \ell_h^ - - \ell]$ implies
\begin{equation}\label{eq:Leprous}
g_h(\ell) - g_h(\ell_h^\pm) = v_h \, (\ell_h^+ - \ell) \, (\ell - \ell_h^-),
\end{equation}
and therefore \eqref{eq:Lorenzo4} becomes
\begin{align}\label{eq:quadratic}
&\delta_h \, D\left(\phi_h(\xi)\right) \, \phi_h'(\xi) =
v_h \left[\ell_h^+-\phi_h(\xi)\right] \left[\phi_h(\xi)-\ell_h^-\right],&
&\xi \in \R.
\end{align}

We first consider stationary traveling waves and specify Theorem \ref{t:1} and Proposition \ref{rem:SC} in the current framework.
We define the intervals
\begin{align*}
&\mathcal{L}_j^0\doteq\begin{cases}
\left[0,1/2\right)&\text{if }\alpha_{1,j} \, v_{1,j} \le 1,
\\
\left[0,\frac{ 1 - \sqrt{1-\alpha_{1,j}^{-1} \, v_{1,j}^{-1}}}{2} \right)&\text{if }\alpha_{1,j} \, v_{1,j} > 1,
\end{cases}&&j \in \mathsf{J}.
\end{align*}

\begin{proposition}
Problem \eqref{eq:model}-\eqref{eq:Queen} admits infinitely many stationary traveling waves; their end states are characterized by the conditions
\begin{align*}
&\ell_1^- \in \bigcap_{j \in \mathsf{J}}\mathcal{L}_j^0,&
&\ell_1^+ + \ell_1^- = 1,&
&\ell_j^\pm = \frac{1}{2} \left[
1\pm\sqrt{1-4 \, \alpha_{1,j} \, v_{1,j} \, \ell_1^+ \, \ell_1^-}
\right],&
&j \in \mathsf{J}.
\end{align*}
Moreover, up to shifts, any stationary traveling wave satisfies \eqref{e:dens_cont}.
\end{proposition}

\begin{proof}
The first part of the proposition follows from Theorem \ref{t:1}.
Indeed, conditions \eqref{e:phinonst}, $\eqref{eq:ch}_1$ and \eqref{basic} are satisfied if and only if for any $h\in\mathsf{H}$ and $j\in\mathsf{J}$
\begin{align*}
&\ell_h^- \in [0, 1/2),&
&\ell_h^+ + \ell_h^- = 1,&
&\ell_j^- \, (1 - \ell_j^-) = \alpha_{1,j} \, v_{1,j} \, \ell_1^- \, (1 - \ell_1^-);
\end{align*}
then it is sufficient to compute $\ell_j^\pm$ and to observe that the definition of $\mathcal{L}_j^0$ guarantees that they are real numbers. 

The latter part of the proposition is deduced by Proposition \ref{rem:SC} because $1/2 \in \emph{S} \doteq \bigcap_{h \in \mathsf{H}}(\ell_h^-, \ell_h^+) \ne \emptyset$. 
\end{proof}

In the following we treat the existence of non-stationary traveling waves. Since $m=1$, by Lemma \ref{l:cj=0} this is equivalent to assume $c_h\ne 0$ for $h \in \mathsf{H}$, namely, the traveling wave is completely non-stationary. 
By \eqref{eq:fjpm}, \eqref{e:Gspeedg} and \eqref{e:phinonst}, from \eqref{eq:kappaj} we deduce
\begin{equation}\label{eq:c1jkappa}
\begin{aligned}
c_{1,j}&=v_{1,j}\frac{1-\ell_1^+-\ell_1^-}{1-\ell_j^+-\ell_j^-},& A_{1,j}&=\alpha_{1,j}v_{1,j}\frac{1-\ell_1^+-\ell_1^-}{1-\ell_j^+-\ell_j^-},
\\
k_j &=
A_{1,j} \, L_{1,j}^\pm - \ell_j^\pm,& \kappa_j&=v_j\ell_j^-\ell_j^+-\alpha_{1,j}v_1\ell_1^-\ell_1^+.
\end{aligned}
\end{equation}

The following result translates Theorem \ref{thm:1} to the present case.
We define the intervals
\begin{align*}
&\mathcal{L}_j^c\doteq\begin{cases}
[0,1]&\text{if }\alpha_{1,j} \, v_{1,j} \le 1,
\\
[0,1] \setminus \left(\frac{1 - \sqrt{1-\alpha_{1,j}^{-1} \, v_{1,j}^{-1}}}{2} , \frac{1 + \sqrt{1-\alpha_{1,j}^{-1} \, v_{1,j}^{-1}}}{2}\right) &\text{if }\alpha_{1,j} \, v_{1,j} > 1,
\end{cases}&&j \in \mathsf{J}.
\end{align*}

\begin{proposition}\label{p:TheMarsVolta}
Problem \eqref{eq:model}-\eqref{eq:Queen} admits a (completely) non-stationary traveling wave if and only if the following condition holds.
\begin{itemize}
\item[($\mathcal{T}_q$)]  There exist $\ell_1^{\pm}\in [0,1]$ with $\ell_1^-<\ell_1^+$ such that:
\begin{enumerate}

\item[(i)]  $\ell_1^++\ell_1^-\ne 1$;

\item[(ii)] $\ell_1^\pm \in \bigcap_{j \in \mathsf{J}}\mathcal{L}_j^c$;

\item[(iii)] for any $j \in \mathsf{J}$ we have
\begin{align}\label{eq:maraton}
&D(\ell)= \, \frac{\alpha_{1,j} \, \delta_{1,j}}{v_{1,j}} \, D\left(\frac{\ell+k_j}{A_{1,j}}\right),&
&\ell\in(\ell_j^-,\ell_j^+),
\end{align}
where $k_j$ is defined in \eqref{eq:c1jkappa} with $\ell_j^\pm$ being solutions to
\begin{equation}\label{eq:fjpmGspeed}
\ell_j^{\pm}(1-\ell_j^{\pm})=\alpha_{1,j}v_{1,j}L_{1,j}^\pm(1-L_{1,j}^\pm).
\end{equation}
\end{enumerate}
\end{itemize}
\end{proposition}

\begin{proof}
The proof consists in showing that, in the present case, condition ($\mathcal{T}$) of Theorem \ref{thm:1} is equivalent to ($\mathcal{T}_q$).
\medskip\\
$\bullet$~The first item of ($\mathcal{T}$) is clearly equivalent to the first item of ($\mathcal{T}_q$).
\medskip\\
$\bullet$~We prove now that the second item of ($\mathcal{T}$) is equivalent to the second item of ($\mathcal{T}_q$).

\smallskip
\noindent
``$\Rightarrow$''
Assume that for any $j \in \mathsf{J}$ there exist $\ell_j^\pm\in[0,1]$ satisfying \eqref{eq:fjpm} and such that $f_j(\ell_j^-) \ne f_j(\ell_j^+)$.
Fix $j \in \mathsf{J}$.
Clearly \eqref{eq:fjpm} is equivalent to \eqref{eq:fjpmGspeed}.

If we denote $z_{1,j}^\pm \doteq 4\alpha_{1,j}v_{1,j}L_{1,j}^\pm(1-L_{1,j}^\pm)$, then the $\ell_j^\pm$-solutions to \eqref{eq:fjpmGspeed} are, see \figurename~\ref{fig:4cases},
\begin{align}\label{eq:solfjpmGspeed}
&\begin{cases}
\ell_j^- = \frac{1}{2} \left[1-\sqrt{1-z_{1,j}^-}\right],
\\
\ell_j^+ \in \left\{
\frac{1}{2} \left[1\pm\sqrt{1-z_{1,j}^+}\right]
\right\},
\end{cases}
\hbox{ if }c_j>0,&
&\begin{cases}
\ell_j^- \in \left\{
\frac{1}{2} \left[1\pm\sqrt{1-z_{1,j}^-}\right]
\right\},
\\
\ell_j^+ = \frac{1}{2} \left[1+\sqrt{1-z_{1,j}^+}\right],
\end{cases}
\hbox{ if }c_j<0.
\end{align}
The square roots in \eqref{eq:solfjpmGspeed} are real numbers if and only if $z_{1,j}^\pm\le1$, namely,
\[\ell_1^\pm(1-\ell_1^\pm) \le (4\alpha_{1,j}v_{1,j})^{-1}.\]
It is easy to see that the above estimate is equivalent to require $\ell_1^\pm \in \mathcal{L}_j^c$.

\smallskip
\noindent
``$\Leftarrow$''
Assume that $\ell_1^\pm \in \bigcap_{j \in \mathsf{J}}\mathcal{L}_j^c$ and fix $j \in \mathsf{J}$.
The square roots in \eqref{eq:solfjpmGspeed} are then real numbers and $\ell_j^\pm$ given in \eqref{eq:solfjpmGspeed} satisfy \eqref{eq:fjpmGspeed}, namely \eqref{eq:fjpm}.
Obviously $\ell_j^\pm$ belong to $[0,1]$.
Finally, since $\ell_j^\pm$ are solutions to \eqref{eq:fjpmGspeed}, it is easy to see that $f_j(\ell_j^+) \ne f_j(\ell_j^-)$ because $f_1(\ell_1^+) \ne f_1(\ell_1^-)$.
\medskip\\
$\bullet$~We prove now that ($\mathcal{T}$) implies the last item of ($\mathcal{T}_q$).
Since the first two items in ($\mathcal{T}$) are equivalent to the first two items in ($\mathcal{T}_q$), we can assume that $\ell_1^++\ell_1^-\ne 1$, $\ell_1^\pm \in \bigcap_{j \in \mathsf{J}}\mathcal{L}_j^c$ and that for any $j \in \mathsf{J}$ we have \eqref{eq:RATM}, namely,
\begin{equation}\label{e:GS_equiv2}
\dfrac{\left(\ell_j^+-\ell_j(c_j \, \xi)\right) \left(\ell_j(c_j \, \xi)-\ell_j^-\right)}{D\left(\ell_j(c_j \, \xi)\right)}
=
\frac{A_{1,j} \, c_{1,j} \, v_{1,j}}{\delta_{1,j}} \,
\dfrac{\left(\ell_1^+-\phi_1( c_1 \, \xi)\right)  \left(\phi_1( c_1 \, \xi)-\ell_1^-\right)}{D\left(\phi_1( c_1 \, \xi)\right)}
\end{equation}
for a.e.\ $\xi \in \R$, where $\phi_1$ is a solution to \eqref{e:infty}-\eqref{eq:quadratic} and
\begin{align*}
&\ell_j(\xi) \doteq A_{1,j} \left[ \phi_1(c_{1,j}\, \xi)-L_{1,j}^{\pm} \right]+ \ell_j^{\pm},&
&\xi\in\R.
\end{align*}
We point out that the above expression of $\ell_j$ is deduced from \eqref{e:elljn} by applying \eqref{eq:c1jkappa}; moreover \eqref{e:GS_equiv2} is deduced from \eqref{eq:RATM} by applying \eqref{eq:Leprous}.
Recall that both fractions in \eqref{e:GS_equiv2} are meant as in \eqref{e:gammah}.
Since
\begin{align*}
\left(\ell_j^+-\ell_j(c_j \, \xi)\right) \left(\ell_j(c_j \, \xi)-\ell_j^-\right) &=
A_{1,j}^2 \left( L_{1,j}^+ - \phi_1(c_1\, \xi) \right) \left( \phi_1(c_1\, \xi) - L_{1,j}^- \right),
\\
\left(\ell_1^+-\phi_1( c_1 \, \xi)\right)  \left(\phi_1( c_1 \, \xi)-\ell_1^-\right) &=
\left(L_{1,j}^+-\phi_1( c_1 \, \xi)\right)  \left(\phi_1( c_1 \, \xi)-L_{1,j}^-\right),
\end{align*}
we have that \eqref{e:GS_equiv2} is equivalent to
\begin{align*}
&D\left(\ell_j(c_j\xi)\right) = \dfrac{\alpha_{1,j}\delta_{1,j}}{v_{1,j}} D\left(\phi_1(c_1 \, \xi)\right)&&\text{for a.e.\ }\xi \in \R.
\end{align*}
To conclude now that the above condition is equivalent to \eqref{eq:maraton} it is sufficient to recall that by Lemma \ref{l:phij} the continuous function $\xi \mapsto \ell_j(\xi)$ is increasing and $\ell_j(\pm\infty) = \ell_j^\pm$ and that $\ell_j(\xi) = A_{1,j} \, \phi_1(c_{1,j} \, \xi) - k_j$ by \eqref{eq:c1jkappa}.
\medskip\\
$\bullet$~Finally, to prove that ($\mathcal{T}_q$) implies the last item of ($\mathcal{T}$) it is enough to trace backwards the proof of the previous item.
\end{proof}

We notice that if $D$ is a polynomial with degree $\mathfrak{d}$, then \eqref{eq:maraton} is equivalent to $\mathfrak{d} + 1$ conditions on the parameters, see for instance \eqref{eq:Anneke} and \eqref{eq:mnk-0}.

\begin{remark}\label{rem:AmonAmarth}
We point out that by Proposition \ref{p:super} we have that problem \eqref{eq:model}-\eqref{eq:Queen} admits a (completely) non-stationary traveling wave satisfying \eqref{e:dens_cont} if and only if
\begin{align}\label{eq:AmonAmarth}
&v_{1,j}^2=\delta_{1,j} &\hbox{ and }&
&\alpha_{1,j} v_{1,j} = 1,&
&j\in\mathsf{J}.
\end{align}
\end{remark}

The special cases of constant or linear diffusivities are treated in the following subsections.

\subsection{The case of constant diffusivities}\label{s:fGdC}

In this subsection we assume
\begin{equation}\label{e:Dhconstant}
D \doteq 1,
\end{equation}
and in this case problem \eqref{e:infty}-\eqref{eq:quadratic} reduces to
\begin{equation}\label{e:GGproblem}
\begin{cases}
\delta_h \, \phi_h'(\xi) = v_h \left[\ell_h^+-\phi_h(\xi)\right]\left[\phi_h(\xi)-\ell_h^-\right],&\xi \in \R,
\\
\phi_h(\pm\infty) = \ell_h^\pm.
\end{cases}
\end{equation}
For any $h\in\mathsf{H}$, the function
\begin{equation}\label{eq:Satriani}
\psi_h(\xi)
\doteq
\dfrac{\ell_h^+}{1+e^{-\frac{v_h}{\delta_h} \left[\ell_h^+-\ell_h^-\right] \xi}}
+
\dfrac{\ell_h^-}{1+e^{\frac{v_h}{\delta_h} \left[\ell_h^+-\ell_h^-\right] \xi}}
\end{equation}
solves \eqref{e:GGproblem} because $\ell_h^- <\ell_h^+$; all the other solutions are of the form $\phi_h(\xi)=\psi_h(\xi+\sigma_h)$ for $\sigma_h \in \R$. Notice that $\psi_h(0)=(\ell_h^+ + \ell_h^-)/2$. 

We rewrite Proposition \ref{p:TheMarsVolta} in the current setting; we emphasize that the shifts appear below because in this case we have the {\em explicit} solution \eqref{eq:Satriani} to problem \eqref{e:GGproblem}.

\begin{proposition}\label{p:GGfinal}
Assume \eqref{e:Dhconstant}. 
Problem \eqref{eq:model}-\eqref{eq:Queen} admits a (completely) non-stationary traveling wave if and only if
\begin{equation}\label{eq:Anneke}
\alpha_{1,j}\delta_{1,j}=v_{1,j}.
\end{equation}
In this case any non-stationary traveling wave $\rho$ has a profile $\varphi$ of the form
\begin{align}\label{e:phiphij}
&\phi(\xi)=\left(\psi_1(\xi +\sigma_1), \ldots, \psi_{n+1}(\xi+\sigma_{n+1})\right),&
&\xi \in \R,
\end{align}
with $\ell_h^\pm$ satisfying (i), (ii) and \eqref{eq:fjpmGspeed} in Proposition \ref{p:TheMarsVolta} and $\sigma_h\in \R$, $h\in\mathsf{H}$, such that
\begin{align}\label{eq:Anneke2}
&c_j\sigma_1 =c_1\sigma_j,&&j\in\mathsf{J}.
\end{align}
\end{proposition}

\begin{proof}
By Theorem \ref{t:E}, any solution to \eqref{e:GGproblem} has the form \eqref{e:phiphij} with $\sigma_h\in \R$, $h\in\mathsf{H}$. Therefore, by Proposition \ref{p:main} it only remains to prove that \eqref{eq:Tool} is equivalent to \eqref{eq:Anneke}-\eqref{eq:Anneke2}.
Straightforward computations show that in the present case \eqref{eq:Tool} can be written as
\begin{align}\label{eq:DMB}
&\frac{f_j(\ell_j^+) \, \zeta_j(t) + f_j(\ell_j^-)}{1+\zeta_j(t)} = \alpha_{1,j} \, \frac{f_1(\ell_1^+) \, \zeta_1(t) + f_1(\ell_1^-)}{1+\zeta_1(t)},&t \in \R, ~ j \in \mathsf{J},
\end{align}
where $\zeta_h(t) \doteq \exp z_h(t)$, for $z_h(t) \doteq \frac{v_h}{\delta_h} \, (\ell_h^+-\ell_h^-) (c_ht+\sigma_h)$, $h \in \mathsf{H}$.
By Proposition \ref{p:Tool} we have
\begin{align*}
&\hbox{ either }\quad f_j(\ell_j^\pm) = \alpha_{1,j}f_1(\ell_1^\pm),&
&\hbox{ or } \quad f_j(\ell_j^\pm) = \alpha_{1,j}f_1(\ell_1^\mp).
\end{align*}
\medskip\noindent$\bullet$~In the former case, identity \eqref{eq:DMB} is equivalent to
\begin{align*}
&\left[ f_j(\ell_j^+) - f_j(\ell_j^-) \right] \left[\zeta_j(t)-\zeta_1(t) \right]=0,
&t \in \R, ~ j \in \mathsf{J}.
\end{align*}
Since by assumption $f_j(\ell_j^+) \neq f_j(\ell_j^-)$, it must be $\zeta_j \equiv \zeta_1$, i.e., $z_j(t) = z_1(t)$, namely
\begin{align*}
&\begin{cases}
\frac{v_j}{\delta_j} \, (\ell_j^+-\ell_j^-) \, c_j = \frac{v_1}{\delta_1} \, (\ell_1^+-\ell_1^-) \, c_1,
\\
\frac{v_j}{\delta_j} \, (\ell_j^+-\ell_j^-) \, \sigma_j = \frac{v_1}{\delta_1} \, (\ell_1^+-\ell_1^-) \, \sigma_1,
\end{cases}&
&\Leftrightarrow&
&\begin{cases}
\frac{v_{1,j}}{\delta_{1,j}} = \frac{f_j(\ell_j^+) - f_j(\ell_j^-)}{f_1(\ell_1^+) - f_1(\ell_1^-)} = \alpha_{1,j},
\\
\frac{\sigma_j}{c_j} = \frac{\sigma_1}{c_1}.
\end{cases}&
\end{align*}
\\$\bullet$~In the latter case, identity \eqref{eq:DMB} is equivalent to
\begin{align*}
&\left[ f_j(\ell_j^+) - f_j(\ell_j^-) \right]\left[\zeta_j(t)\zeta_1(t)-1\right]=0,
&t \in \R, ~ j \in \mathsf{J}.
\end{align*}
Since by assumption $f_j(\ell_j^+) \neq f_j(\ell_j^-)$, it must be $\zeta_j \, \zeta_1 \equiv 1$, i.e. $z_j(t) = -z_1(t)$, namely
\begin{align*}
&\begin{cases}
\frac{v_j}{\delta_j} \, (\ell_j^+-\ell_j^-) \, c_j =- \frac{v_1}{\delta_1} \, (\ell_1^+-\ell_1^-) \, c_1,
\\
\frac{v_j}{\delta_j} \, (\ell_j^+-\ell_j^-) \, \sigma_j =- \frac{v_1}{\delta_1} \, (\ell_1^+-\ell_1^-) \, \sigma_1,
\end{cases}&
&\Leftrightarrow&
&\begin{cases}
\frac{v_{1,j}}{\delta_{1,j}} = -\frac{f_j(\ell_j^+) - f_j(\ell_j^-)}{f_1(\ell_1^+) - f_1(\ell_1^-)} = \alpha_{1,j},
\\
\frac{\sigma_j}{c_j} = \frac{\sigma_1}{c_1}.
\end{cases}&
\end{align*}
In both cases we proved that \eqref{eq:Tool} is equivalent to \eqref{eq:Anneke}-\eqref{eq:Anneke2}; this concludes the proof.
\end{proof}

\begin{remark}
Consider conditions $\eqref{eq:AmonAmarth}_1$, $\eqref{eq:AmonAmarth}_2$ and \eqref{eq:Anneke}. Any two of them implies the third one.
\end{remark}

\begin{proposition}
Assume \eqref{e:Dhconstant}. 
Problem \eqref{eq:model}-\eqref{eq:Queen} admits a (completely) non-stationary traveling wave satisfying \eqref{e:dens_cont} if and only if \eqref{eq:AmonAmarth} holds true.
In this case a non-stationary traveling wave satisfies \eqref{e:dens_cont} if and only if its end states satisfy \eqref{eq:Beardfish}.
\end{proposition}

\begin{proof}
The first part of the statement is just Remark \ref{rem:AmonAmarth}. In this case, since \eqref{eq:AmonAmarth} implies \eqref{eq:Anneke}, by Proposition \ref{p:GGfinal} any (completely) non-stationary traveling wave $\rho$ has a profile of the form \eqref{e:phiphij}-\eqref{eq:Anneke2}.

The second part of the statement characterizes the end states. If a non-stationary traveling wave $\rho$ satisfies \eqref{e:dens_cont}, then \eqref{eq:Beardfish} holds because of Lemma \ref{l:+cc}. Conversely, if the end states of $\rho$ satisfy \eqref{eq:Beardfish}, then long but straightforward computations show that \eqref{e:cont_profiles} holds true, and therefore $\rho$ satisfies \eqref{e:dens_cont}. 
\end{proof}


\subsection{The case of linear diffusivities}\label{s:fGdL}
In this subsection we assume
\begin{equation}\label{eq:d-lin}
D(\rho) \doteq \rho.
\end{equation}
We notice that $D$ degenerates at $0$ and this makes the subject more interesting. In this case problem \eqref{e:infty}-\eqref{eq:quadratic} reduces to
\begin{equation}\label{e:GDlinear}
\begin{cases}
\delta_h\varphi_{h}\varphi'_h =
v_h (\ell^{+}_h - \varphi_h) (\varphi_h - \ell^{-}_h),& \xi \in \R,
\\
\varphi_{h}(\pm\infty)=\ell_h^{\pm}.
\end{cases}
\end{equation}
If $\ell_h^-=0$, then the function
\begin{equation}\label{eq:phi2}
\psi_h(\xi) \doteq 
\begin{cases}
\frac{\ell_h^+}{2}\left(2-e^{-\frac{v_h}{\delta_h}\xi}\right)&\hbox{ if }\xi\ge-\frac{\delta_h}{v_h}\ln2,
\\[2mm]
0&\hbox{ if }\xi<-\frac{\delta_h}{v_h}\ln2,
\end{cases}
\end{equation}
solves \eqref{e:GDlinear} because $\ell_h^- <\ell_h^+$; by \eqref{e:Ih} we have $I_h=\left(-\frac{\delta_h}{v_h}\ln2,\infty\right)$. If $\ell_h^->0$, then $I_h=\R$ and the function $\psi_h$ implicitly given by
\begin{equation}\label{eq:implicitly}
\left[
2 \exp\left( \dfrac{v_h}{\delta_h} \, \xi\right) \frac{\psi_h(\xi)-\ell_h^-}{\ell_h^+-\ell_h^-}
\right]^{\ell_h^-}
=
\left[
2 \exp\left( \dfrac{v_h}{\delta_h} \, \xi\right) \frac{\ell_h^+-\psi_h(\xi)}{\ell_h^+-\ell_h^-}
\right]^{\ell_h^+}
\end{equation}
solves \eqref{e:GDlinear} because $\ell_h^- <\ell_h^+$.
Notice that in both cases $\psi_h(0)=(\ell_h^+ + \ell_h^-)/2$ and all the other solutions are of the form $\phi_h(\xi)=\psi_h(\xi+\sigma_h)$ for $\sigma_h \in \R$.
Hence, any non-stationary traveling wave $\rho$ has a profile $\varphi$ of the form
\begin{align}
&\phi(\xi)=\left(\psi_1(\xi +\sigma_1), \ldots, \psi_{n+1}(\xi+\sigma_{n+1})\right),&
&\xi \in \R.
\label{e:phiphij2}
\end{align}
In the sequel we prove that the shifts $\sigma_h$, $h \in \mathsf{H}$, satisfy \eqref{eq:Anneke2}, or equivalently
\begin{align}\label{eq:PainOfSalvation}
&v_{1,j} \, \sigma_1 =\delta_{1,j} \, \sigma_j,&&j\in\mathsf{J}.
\end{align}

\begin{lemma}\label{l:CNSGreen}
Assume \eqref{eq:d-lin}. If $\ell_1^+ + \ell_1^-\ne1$, then condition \eqref{eq:maraton} is equivalent to
\begin{align}\label{eq:mnk-0}
&\frac{v_{1,j}^2}{\delta_{1,j}} = \frac{1-\ell_j^+-\ell_j^-}{1-\ell_1^+-\ell_1^-}&\hbox{ and }&
&\ell_j^- \, \ell_j^+ = \alpha_{1,j} \, v_{1,j} \, \ell_1^- \, \ell_1^+.
\end{align}
\end{lemma}
\begin{proof}
In the present case, condition \eqref{eq:maraton} becomes
$(c_{1,j}v_{1,j} - \delta_{1,j}) \, \ell - \delta_{1,j} \, k_j = 0$ for $\ell \in (\ell_j^-,\ell_j^+)$:
it is satisfied if and only if both $c_{1,j}v_{1,j} =  \delta_{1,j}$ and $k_j=0$.
The former is equivalent to  \eqref{eq:mnk-0}$_1$, the latter is equivalent to \eqref{eq:mnk-0}$_2$ by \eqref{eq:c1jkappa}$_4$, because $\kappa_j=0$.
\end{proof}

We observe that \eqref{eq:mnk-0}$_1$ and \eqref{eq:c1jkappa}$_1$ imply that $c_{1,j} = \delta_{1,j}/v_{1,j} > 0$; therefore \eqref{eq:fjpmGspeed} becomes
\begin{align}\label{eq:linear}
&\ell_j^{\pm}(1- \ell_j^{\pm})=\alpha_{1,j}\, v_{1,j}\, \ell_1^{\pm}(1- \ell_1^{\pm}), && j \in \mathsf{J}.
\end{align}
As a consequence $\rho$ is either non-degenerate or completely degenerate.

Now, we discuss (completely) non-stationary traveling waves by considering separately the (completely) degenerate and non-degenerate case.
We denote
\begin{align*}
&\Delta_j \doteq \left\{ \alpha_{1,j}\, \delta_{1,j}, \,\, \sqrt{\delta_{1,j}}, \,\, \sqrt[3]{\alpha_{1,j}\, \delta^2_{1,j}} \right\},&
&j \in \mathsf{J}.
\end{align*}
\begin{proposition}\label{p:beta2delta}
Assume \eqref{eq:d-lin}.
Problem \eqref{eq:model}-\eqref{eq:Queen} admits a traveling wave that is both (completely) degenerate and (completely) non-stationary if and only if either \eqref{eq:AmonAmarth} holds true or
\begin{equation}
\label{eq:beta-max}
\begin{array}{c}
\begin{aligned}
0 < &v_{1,j} < \min \Delta_j &\text{or}\ & &v_{1,j} > \max \Delta_j,& &j \in \mathsf{J},
\end{aligned}
\\[5pt]
\dfrac{v_{1,2} (\delta_{1,2}-v_{1,2}^2)}{\alpha_{1,2} \, \delta_{1,2}^2-v_{1,2}^3} =
\ldots =
\dfrac{v_{1,n+1} (\delta_{1,n+1}-v_{1,n+1}^2)}{\alpha_{1,n+1} \, \delta_{1,n+1}^2-v_{1,n+1}^3}.
\end{array}
\end{equation}

In the first case, problem \eqref{eq:model}-\eqref{eq:Queen} has infinitely many of such waves; each of them satisfies \eqref{eq:Beardfish} and (up to shifts) \eqref{e:dens_cont}.

In the second case, problem \eqref{eq:model}-\eqref{eq:Queen} has a unique (up to shifts) such wave, which does not satisfy (for no shifts) \eqref{e:dens_cont}. Its end states do not satisfy \eqref{eq:Beardfish} and are
\begin{align}\label{eq:ParadiseLost}
&\ell_1^- = 0 = \ell_j^-,&
&\ell_1^+ = \frac{v_{1,j} (\delta_{1,j}-v_{1,j}^2)}{\alpha_{1,j} \, \delta_{1,j}^2-v_{1,j}^3},&
&\ell_j^+ = \alpha_{1,j} \, \frac{\delta_{1,j} (\delta_{1,j}-v_{1,j}^2)}{\alpha_{1,j} \, \delta_{1,j}^2-v_{1,j}^3},& &j \in \mathsf{J}.
\end{align}
In both cases, any degenerate non-stationary traveling wave $\rho$ has a profile $\varphi$ of the form \eqref{e:phiphij2} with $\psi_h$ defined by \eqref{eq:phi2} and $\sigma_h\in \R$, $h\in\mathsf{H}$, satisfying \eqref{eq:PainOfSalvation}.
\end{proposition}

\begin{proof}
We claim that the existence of a degenerate non-stationary traveling wave is equivalent to the existence of $\ell_h^+\in (0,1), \, h \in \mathsf{H}$, such that
\begin{align}\label{e:values}
&\ell^+_j=\alpha_{1,j} \, \frac{\delta_{1,j}}{v_{1,j}} \, \ell_1^+&
&\hbox{ and }&
&\left[ \alpha_{1,j} \, \delta_{1,j}^2 - v_{1,j}^3 \right] \ell_1^+
+ v_{1,j} \left[v_{1,j}^2 - \delta_{1,j}\right]  = 0,&
&j\in \mathsf{J}.
\end{align}
In fact, by Proposition~\ref{p:TheMarsVolta} the existence of a non-stationary traveling wave is equivalent to condition ($\mathcal{T}_q$), where \eqref{eq:maraton} can be written as \eqref{eq:mnk-0} by Lemma \ref{l:CNSGreen} and \eqref{eq:fjpmGspeed} as \eqref{eq:linear}.
Then, \eqref{eq:mnk-0} and \eqref{eq:linear} with $\ell_h^-=0$, $h \in \mathsf{H}$, reduce to the relation among the end states
\begin{align}\label{eq:Deftones}
&\frac{v_{1,j}^2}{\delta_{1,j}}=\frac{1-\ell_j^+}{1-\ell_1^+},&
&\ell_j^+ \, (1 - \ell_j^+) = \alpha_{1,j} \, v_{1,j} \, \ell_1^+ \, (1 - \ell_1^+),&
&j \in\mathsf{J}.
\end{align}
By \eqref{eq:Deftones} we obtain $v_{1,j}^2/\delta_{1,j} = \alpha_{1,j} \, v_{1,j} \, \ell_1^+/\ell_j^+$ and then $\eqref{e:values}_1$; by plugging $\eqref{e:values}_1$ into $\eqref{eq:Deftones}_2$ we get $\eqref{e:values}_2$ and then the claim.

Assume there is a degenerate non-stationary traveling wave; then $\ell_1^+$ satisfies \eqref{e:values}$_2$. As a consequence, we have either $\alpha_{1,j}\delta^2_{1,j}- v^3_{1,j}=v_{1,j}^2-\delta_{1,j}=0$ or $\alpha_{1,j}\delta^2_{1,j}\ne v^3_{1,j}$ for every $j \in \mathsf{J}$. The former case is equivalent to  \eqref{eq:AmonAmarth}. In the latter case we can explicitly compute $\ell_1^+$ by $\eqref{e:values}_2$ for any $j \in \mathsf{J}$ and impose the constraint $0<\ell_1^+<1$, namely,
\begin{align*}
0< \dfrac{v_{1,j} \, (\delta_{1,j}-v_{1,j}^2)}{\alpha_{1,j} \, \delta_{1,j}^2-v_{1,j}^3}  < 1.
\end{align*}
A direct computation shows that this is equivalent to \eqref{eq:beta-max}. In conclusion, either condition \eqref{eq:AmonAmarth} or \eqref{eq:beta-max} is necessary for the existence of a non-stationary traveling wave with $\ell_1^-=0$.

Conversely, assume condition \eqref{eq:AmonAmarth}. In this case $\alpha_{1,j} \, \delta_{1,j}^2 =\alpha_{1,j} \, v_{1,j}^4=v_{1,j}^3$. Then \eqref{e:values}$_2$ is trivially satisfied for every $\ell_1^+\in (0,1)$ and from \eqref{e:values}$_1$  we deduce $\ell_1^+=\ell_j^+$. Hence, there is an infinite family of non-stationary traveling waves parameterized by $\ell_1^+ \in (0,1)$ and satisfying \eqref{eq:Beardfish}; as a consequence, they do not coincide up to shifts and they all satisfy (up to shifts) the continuity condition \eqref{e:dens_cont}.

Assume now condition \eqref{eq:beta-max}. In this case the values for $\ell_1^+$ and $\ell_j^+$ in \eqref{eq:ParadiseLost} are well defined since $v^3_{1,j}\ne \alpha_{1,j}\delta_{1,j}$ and they are the unique solution to \eqref{e:values}. In particular, condition {\em (ii)} in Proposition \ref{p:TheMarsVolta} is automatically satisfied. By the estimates in \eqref{eq:beta-max} we have $\ell_1^+,\ell_j^+ \in (0,1)$ for $j \in \mathsf{J}$. Hence, there is a unique (up to shifts) degenerate non-stationary traveling wave and its end states satisfy \eqref{eq:ParadiseLost}. Furthermore, by Lemma \ref{l:+cc}, condition \eqref{e:dens_cont} implies  \eqref{eq:Beardfish}, which is  precluded by \eqref{eq:beta-max}. Hence, the traveling wave does not satisfy \eqref{e:dens_cont}.

At last, by Theorem \ref{t:E}, any solution to \eqref{e:GDlinear} has the form \eqref{e:phiphij2}.
By \eqref{e:squares}, that in the present case becomes
\begin{align*}
\phi_j'(c_j\xi) = \alpha_{1,j} c_{1,j}^2 \phi_1'\left(c_1\xi\right)&
&\hbox{ for a.e.\ }\xi\in\R,\ j \in \mathsf{J},
\end{align*}
and the regularity of $\psi_h$ defined in \eqref{eq:phi2}, we have
\[
\frac{1}{c_j} \left[ \frac{\delta_j}{v_j} \, \ln 2 + \sigma_j \right]
=
\frac{1}{c_1} \left[ \frac{\delta_1}{v_1} \, \ln 2 + \sigma_1 \right],
\]
which is equivalent to \eqref{eq:Anneke2} because $c_{1,j} = \delta_{1,j}/v_{1,j}$.
\end{proof}

The following result treats the non-degenerate case.

\begin{proposition}
Assume \eqref{eq:d-lin}.
Problem \eqref{eq:model}-\eqref{eq:Queen} admits a non-degenerate (completely) non-stationary traveling wave if and only if condition \eqref{eq:AmonAmarth} is satisfied. In this case any non-degenerate non-stationary traveling wave satisfies (up to shifts) \eqref{e:dens_cont}; moreover, it has a profile $\varphi$ of the form \eqref{e:phiphij2} with $\psi_h$ implicitly defined by \eqref{eq:implicitly} and $\sigma_h\in \R$, $h\in\mathsf{H}$, satisfying \eqref{eq:PainOfSalvation}.
\end{proposition}
\begin{proof}
Assume that there is a non-degenerate non-stationary traveling wave; 
then $\ell_h^- \ne 0$ and $1 \ne \ell_h^+ + \ell_h^-$, $h \in \mathsf{H}$. Moreover, by Proposition~\ref{p:TheMarsVolta}, condition ($\mathcal{T}_q$) is satisfied, where \eqref{eq:maraton} becomes \eqref{eq:mnk-0} by Lemma \ref{l:CNSGreen} and \eqref{eq:fjpmGspeed} is \eqref{eq:linear}.
When dividing \eqref{eq:linear} by \eqref{eq:mnk-0}$_2$ we obtain
\begin{align*}
&(1-\ell_j^+) \, \ell_1^- = (1-\ell_1^+) \, \ell_j^-&
&\hbox{ and }&
&(1-\ell_j^-) \, \ell_1^+ = (1-\ell_1^-) \, \ell_j^+,&
&j \in \mathsf{J}.
\end{align*}
By adding the above relations we have $\ell_1^- + \ell_1^+ = \ell_j^- + \ell_j^+$, hence
\[
0 = \ell_1^+ - \ell_j^+ + \ell_1^- - \ell_j^-
= \ell_1^+ - 1 + (1-\ell_1^+) \, \frac{\ell_j^-}{\ell_1^-} + \ell_1^- - \ell_j^-
= \frac{1-\ell_1^+-\ell_1^-}{\ell_1^-} \, (\ell_j^- - \ell_1^-).
\]
It is now easy to conclude that \eqref{eq:Beardfish} is satisfied and then also \eqref{eq:AmonAmarth} holds true by \eqref{eq:mnk-0}. At last, the traveling wave satisfies (up to shifts) \eqref{e:dens_cont} by Remark \ref{rem:AmonAmarth}.

Conversely, assume \eqref{eq:AmonAmarth}. 
Then \eqref{eq:mnk-0} and \eqref{eq:linear} write
\begin{align*}
&\ell_j^+ + \ell_j^- = \ell_1^+ + \ell_1^-, &
&\ell_j^- \, \ell_j^+ =\ell_1^- \, \ell_1^+, &
&\ell_j^\pm(1- \ell_j^\pm)=\ell_1^\pm(1- \ell_1^\pm),&& j \in \mathsf{J}.
\end{align*}
The same computations as before give that if we impose $\ell_h^- \ne 0$ and $1 \ne \ell_h^+ + \ell_h^-$, $h \in \mathsf{H}$, then the above conditions are equivalent to \eqref{eq:Beardfish}; the existence of infinitely many non-degenerate non-stationary traveling waves satisfying \eqref{e:dens_cont} easily follows.

At last, by Theorem \ref{t:E}, any solution to \eqref{e:GDlinear} has the form \eqref{e:phiphij2}.
Fix $j \in \mathsf{J}$. 
By \eqref{e:dens_cont} we have \eqref{e:cont_profiles}, namely
\begin{align*}
&\psi_j(c_jt+\sigma_j) = \psi_1(c_1t+\sigma_1),&&t \in \R.
\end{align*}
This identity together with \eqref{eq:implicitly} and \eqref{eq:Beardfish} imply
\begin{align*}
&\left[
2 \exp\left( \dfrac{v_j}{\delta_j} \, (c_jt+\sigma_j) \right) \frac{\psi_1(c_1t+\sigma_1)-\ell^-}{\ell^+-\ell^-}
\right]^{\ell^-}
=
\left[
2 \exp\left( \dfrac{v_j}{\delta_j} \, (c_jt+\sigma_j) \right) \frac{\ell^+-\psi_1(c_1t+\sigma_1)}{\ell^+-\ell^-}
\right]^{\ell^+},
\\
&\left[
2 \exp\left( \dfrac{v_1}{\delta_1} \, (c_1t+\sigma_1) \right) \frac{\psi_1(c_1t+\sigma_1)-\ell^-}{\ell^+-\ell^-}
\right]^{\ell^-}
=
\left[
2 \exp\left( \dfrac{v_1}{\delta_1} \, (c_1t+\sigma_1) \right) \frac{\ell^+-\psi_1(c_1t+\sigma_1)}{\ell^+-\ell^-}
\right]^{\ell^+}.
\end{align*}
By dividing the above equalities and taking the logarithm we get
\begin{align*}
&\left[ \dfrac{v_j}{\delta_j} \, (c_jt+\sigma_j) - \dfrac{v_1}{\delta_1} \, (c_1t+\sigma_1) \right] \ell^- = \left[ \dfrac{v_j}{\delta_j} \, (c_jt+\sigma_j) - \dfrac{v_1}{\delta_1} \, (c_1t+\sigma_1) \right] \ell^+,&&t \in \R.
\end{align*}
Since $\ell^- \ne \ell^+$ and $c_{1,j} = \delta_{1,j}/v_{1,j}$, the above equality is equivalently to \eqref{eq:PainOfSalvation}.
\end{proof}


\section{Application to the case of a logarithmic flux, \texorpdfstring{$m=1$}{}}\label{s:log-fl}

In this section we assume \eqref{e:proportionality} for some constants $v_h,\delta_h>0$,  $D \doteq 1$ and the logarithmic flux \cite{Greenberg} defined by
\[
f(\rho) \doteq -\rho\ln(\rho)
\]
for $\rho\in(0,1]$ with $f(0)=0$ by continuity; in the following we simply write $\rho \ln (\rho)$ for $\rho \in [0,1]$. We use the notation introduced in \eqref{eq:vdelta}; then, in the present case the diffusivity $D_h$ coincides with the anticipation length $\delta_h$ of \cite{BTTV}, see Section \ref{s:m}. As in Section \ref{s:Greenshields}, we focus on the case $m=1$ and do not mention in the following these assumptions on $f_h$, $D_h$ and $m$.  

Condition \eqref{eq:ch} becomes
\begin{align}\label{eq:chlog}
&0\le\ell_h^- < \ell_h^+\le1&
&\hbox{ and }&
&c_h=-v_h\frac{\ell^+_h\ln(\ell^+_h)-\ell^-_h\ln(\ell^-_h)}{\ell^+_h-\ell^-_h}.
\end{align}
Moreover we have, for $h \in \mathsf{H}$,
\begin{align}\label{eq:ghlog}
&g_{h}(\ell)=
v_h\ell \left[\frac{\ell^+_h\ln(\ell^+_h)-\ell^-_h\ln(\ell^-_h)}{\ell^+_h-\ell^-_h} - \ln(\ell)\right],
\\\nonumber
&g_{h}(\ell) - g_h(\ell_h^{\pm})=
v_h \left[ \frac{(\ell-\ell_h^-)\ell^+_h\ln(\ell^+_h)+(\ell_h^+-\ell)\ell^-_h\ln(\ell^-_h)}{\ell^+_h-\ell^-_h} - \ell\ln(\ell)\right].
\end{align}
Therefore \eqref{eq:Lorenzo4} becomes
\begin{align}\label{eq:equationlog}
&\phi_h'(\xi)=\frac{v_h}{\delta_h}
\left[ \frac{\left[\phi_h(\xi) - \ell_h^-\right] \ell^+_h \ln(\ell^+_h) + \left[\ell_h^+ - \phi_h(\xi)\right] \ell^-_h \ln(\ell^-_h)}{\ell^+_h - \ell^-_h}
- \phi_h(\xi) \ln\left(\phi_h(\xi)\right) \right],
\end{align}
for $\xi \in \R$. Let $\EL:[0,e^{-1}] \to [0,e^{-1}]$ and $\ER:[0,e^{-1}] \to [e^{-1},1]$ be the inverse functions of the restrictions $f_\ell$ and $f_r$ of $f$ to $[0, e^{-1}]$ and $[e^{-1}, 1]$, respectively.

We first consider the case of stationary waves. We define the intervals
\begin{align*}
&\mathcal{L}_j^0\doteq\begin{cases}
[0, e^{-1})&\text{if }\alpha_{1,j} \, v_{1,j} \le 1,
\\
\left[0, \EL(e^{-1}\alpha_{1,j}^{-1} \, v_{1,j}^{-1}) \right)&\text{if }\alpha_{1,j} \, v_{1,j} > 1,
\end{cases}&&j \in \mathsf{J}.
\end{align*}

\begin{proposition}
Problem \eqref{eq:model}-\eqref{eq:Queen} admits infinitely many stationary traveling waves; their end states are characterized by the conditions
\begin{align*}
&\ell_1^- \in \bigcap_{j \in \mathsf{J}}\mathcal{L}_j^0,&
&\ell_1^+ = \ER\left(-\ell_1^-\ln (\ell_1^-)\right),\\
&\ell_j^- = \EL\left(-\alpha_{1,j} \, v_{1,j}\ell_1^-\ln (\ell_1^-) \right),&
&\ell_j^+ = \ER\left(-\alpha_{1,j} \, v_{1,j}\ell_1^-\ln (\ell_1^-)\right),&
&j \in \mathsf{J}.
\end{align*}
Moreover, up to shifts, any stationary traveling wave satisfies \eqref{e:dens_cont}.
\end{proposition}

\begin{proof}
The first part of the proposition follows from Theorem \ref{t:1}.
Indeed, conditions \eqref{eq:ch}$_1$ and  \eqref{basic} are satisfied if and only if for any $h\in\mathsf{H}$ and $j\in\mathsf{J}$ 
\begin{align*}
&\ell_h^- \in [0, e^{-1}),& &\ell_h^-\ln(\ell_h^-)=\ell_h^+\ln(\ell_h^+),&
&\ell_j^- \, \ln( \ell_j^-)= \alpha_{1,j} \, v_{1,j} \, \ell_1^- \, \ln( \ell_1^-).
\end{align*}
Hence $\ell_1^+=\ER\left(-\ell_1^-\ln(\ell_1^-)\right)$ and it is sufficient to determine $\ell_j^\pm$.  Observe that the definition of $\mathcal{L}_j^0$ guarantees that they can be uniquely computed. 
At last, the latter part of the proposition follows by the proof of Proposition \ref{rem:SC} since $e^{-1} \in \emph{S} \doteq \bigcap_{h\in\mathsf{H}}(\ell_h^-,\ell_h^+) \ne \emptyset$.
\end{proof}

In the following  we discuss the existence of non-stationary traveling waves. Since $m=1$, by Lemma \ref{l:cj=0} this is equivalent to assume that  the traveling wave is completely non-stationary. By \eqref{eq:chlog}$_2$ we deduce
\begin{equation}\label{eq:c1jlog}
c_{1,j} = v_{1,j} ~ \frac{\ell_1^+\ln(\ell_1^+)-\ell_1^-\ln(\ell_1^-)}{\ell_j^+\ln(\ell_j^+)-\ell_j^-\ln(\ell_j^-)} ~
\frac{\ell_j^+-\ell_j^-}{\ell_1^+-\ell_1^-}.
\end{equation}
The following result translates Theorem \ref{thm:1} to the current framework. We define the intervals
\begin{align*}
&\mathcal{L}_j^c\doteq\begin{cases}
[0,1]&\text{if }\alpha_{1,j} \, v_{1,j} \le 1,
\\
[0,1] \setminus \left(\EL(e^{-1}\alpha_{1,j}^{-1} \, v_{1,j}^{-1}) , \, \ER(e^{-1}\alpha_{1,j}^{-1} \, v_{1,j}^{-1})\right) &\text{if }\alpha_{1,j} \, v_{1,j} > 1,
\end{cases}&&j \in \mathsf{J}.
\end{align*}

\begin{proposition}\label{p:TheMarsVoltaLOG}
Problem \eqref{eq:model}-\eqref{eq:Queen} admits a (completely) non-stationary traveling wave if and only if the following condition holds.
\begin{itemize}
\item[($\mathcal{T}_l$)]  There exist $\ell_1^{\pm}\in [0,1]$ with $\ell_1^-<\ell_1^+$ such that:
\begin{enumerate}

\item[(i)]  $\ell_1^-\ln(\ell_1^-)\ne \ell_1^+ \ln(\ell_1^+)$;

\item[(ii)] $\ell_1^\pm \in \bigcap_{j \in \mathsf{J}}\mathcal{L}_j^c$;

\item[(iii)] for any $j \in \mathsf{J}$ we have
\begin{align}\label{eq:maratonLOG}
&\delta_{1,j} \left[ g_j(\ell)- g_j(\ell_j^-) \right] = A_{1,j} \, c_{1,j} \left[ g_1\left( \frac{\ell+k_j}{A_{1,j}}\right)-g_1(\ell_1^-) \right],&
&\ell \in (\ell^-_j,\,\ell^+_j),
\end{align}
where $g_h$ is given in \eqref{eq:ghlog}, $c_{1,j}$ in \eqref{eq:c1jlog}, $A_{1,j}$ in \eqref{eq:kappaj}$_2$ and $k_j$ in \eqref{eq:kappaj}$_3$, with $\ell_j^\pm$ being solutions to
\begin{equation}\label{eq:fjpmlog}
\ell_j^{\pm}\ln(\ell_j^{\pm})=\alpha_{1,j}v_{1,j}L_{1,j}^{\pm}\ln(L_{1,j}^{\pm}).
\end{equation}
\end{enumerate}
\end{itemize}
\end{proposition}

\begin{proof} The proof consists in showing that, in the present case, ($\mathcal{T}$) of Theorem \ref{thm:1} is equivalent to ($\mathcal{T}_l$). The first two items in  ($\mathcal{T}$) and  ($\mathcal{T}_l$) are clearly equivalent.
 It remains to discuss the third one. Condition \eqref{eq:RATM} is equivalent to 
\begin{align}\label{eq:CSlog}
&\delta_{1,j}\left[g_j\left(\ell_j(c_j\xi)\right)- g_j(\ell_j^-)\right] = A_{1,j} \, c_{1,j} \left[ g_1\left(\phi_1(c_1\xi)\right)-g_1(\ell_1^-)\right],&
&\xi \in \R,
\end{align}
where $\phi_1$ is a solution to \eqref{e:infty}-\eqref{eq:equationlog} and $\ell_j(\xi) \doteq A_{1,j}\phi_1(c_{1,j}\xi)-k_j$ for $c_{1,j}$ in \eqref{eq:c1jlog}, $A_{1,j}$ in \eqref{eq:kappaj}$_2$ and $k_j$ in \eqref{eq:kappaj}$_3$. By Theorem \ref{t:E}, $\phi_1$ is strictly increasing and so is the function $\ell_j$. Put $\ell \doteq \ell_j(c_j\xi)$. Hence $\ell \in (\ell_j^-, \, \ell_j^+)$, by Lemma \ref{l:phij}, and then \eqref{eq:CSlog} is equivalent to \eqref{eq:maratonLOG}.
\end{proof}

In the following we focus on the case of (completely) non-stationary traveling waves with $\ell_h^-=0$ for some $h \in \mathsf{H}$.

\begin{lemma}\label{lem:special}
Assume that problem \eqref{eq:model}-\eqref{eq:Queen} admits a traveling wave. The following statements are equivalent:
\begin{enumerate}
\item[(i)]
$\ell_1^- = 0$;
\item[(ii)]
$\ell_j^- = 0$ for all $j \in \mathsf{J}$;
\item[(iii)]
there exists ${\rm j} \in \mathsf{J}$ such that $\ell_{\rm j}^- = 0$.
\end{enumerate}
\end{lemma}

\begin{proof}
First, we prove that {\em (i)} implies {\em (ii)}. 
Fix $j \in \mathsf{J}$.
Since $\ell_1^-=0$, then condition \eqref{eq:fjpmlog} implies that either $\ell_j^-=0$ or $\ell_j^+=1$, for $j \in \mathsf{J}$. Assume by contradiction that $\ell_j^+=1$. 
Since $c_{1,j}<0$, condition \eqref{eq:fjpmlog} becomes 
\[
\ell_j^-\ln(\ell_j^-)=\alpha_{1,j}\, v_{1,j}\ell_1^+\ln(\ell_1^+).
\]
Therefore, by \eqref{eq:c1jlog}, \eqref{eq:kappaj}$_2$ and  \eqref{eq:kappaj}$_3$ we have that
\begin{align*}
&c_{1,j} = - v_{1,j} ~ \frac{\ell_1^+\ln(\ell_1^+)}{\ell_j^-\ln(\ell_j^-)} ~
\frac{1-\ell_j^-}{\ell_1^+}
= - \frac{1-\ell_j^-}{\alpha_{1,j} \, \ell_1^+},&
&A_{i,j}  = - \frac{1-\ell_j^-}{\ell_1^+},
&k_j =  -1.
\end{align*}
Condition \eqref{eq:maratonLOG} can be written as
\begin{align*}
&\ell \ln(\ell) - v_{1,j} (1-\ell) \left[ \frac{\alpha_{1,j}\, \ell_1^+\ln(\ell_1^+)}{1-\ell_j^-}
+ \frac{1-\ell_j^-}{\alpha_{1,j} \, \delta_{1,j} \, \ell_1^+}  \, \ln \left( \frac{1-\ell}{1-\ell_j^-}\right) \right]=0,
\end{align*}
for $\ell \in (\ell_j^-, 1)$.
By differentiating the above equation three times we obtain
\begin{align*}
&-\frac{v_{1,j}(1-\ell_j^-)}{\alpha_{1,j} \delta_{1,j} \ell_1^+(1-\ell)^2} = \frac{1}{\ell^2},&
&\ell\in (\ell_j^-, 1).
\end{align*}
This is a contradiction because the two sides have opposite sign. This proves \emph{(ii)}.

Since the implication \emph{(ii)} $\Rightarrow$ \emph{(iii)} is obvious, it remains to show that \emph{(iii)} $\Rightarrow$ \emph{(i)}. Let $\ell_{\rm j}^-=0$ for some ${\rm j} \in \mathsf{J}$. By \eqref{eq:fjpmlog} it follows that either $\ell_1^-=0$ or $\ell_1^+=1$. In the latter case by arguing as above it is easy to obtain a contradiction and then \emph{(iii)} follows.
\end{proof}

At last, we give a result which is similar to the one given in Proposition \ref{p:beta2delta}. We denote 
\begin{align*}
&\Delta_j \doteq \left\{\alpha_{1,j}\, \delta_{1,j}, \sqrt{\delta_{1,j}}\right\},&
&j\in \mathsf{J}.
\end{align*}
By Lemma \ref{lem:special} we have either $\ell_h^-=0$, $h \in \mathsf{H}$, or  $\ell_h^-\ne0$, $h \in \mathsf{H}$.
Below we consider the first case.

\begin{proposition}
Problem \eqref{eq:model}-\eqref{eq:Queen} admits a (completely) non-stationary traveling wave with $\ell_h^-=0$, $h \in \mathsf{H}$, if and only if either \eqref{eq:AmonAmarth} holds true or
\begin{equation}
\label{eq:mnk-1}
\begin{array}{c}
\begin{aligned}
0 < &v_{1,j} < \min \Delta_j &\text{or}\ & &v_{1,j} > \max \Delta_j,& &j \in \mathsf{J},
\end{aligned}
\\[5pt]
\left[\alpha_{1,2} \frac{\delta_{1,2}}{v_{1,2}}\right]^{\frac{\delta_{1,2}}{v_{1,2}^2 - \delta_{1,2}}} =
\ldots =
\left[\alpha_{1,n+1} \frac{\delta_{1,n+1}}{v_{1,n+1}}\right]^{\frac{\delta_{1,n+1}}{v_{1,n+1}^2 - \delta_{1,n+1}}}.
\end{array}
\end{equation}

In the first case, problem \eqref{eq:model}-\eqref{eq:Queen} has infinitely many of such waves; each of them satisfies \eqref{eq:Beardfish} and (up to shifts) \eqref{e:dens_cont}.

In the second case, problem \eqref{eq:model}-\eqref{eq:Queen} has a unique (up to shifts) such wave and such wave, which does not satisfy (for no shifts) \eqref{e:dens_cont}. Its end states are
\begin{align}\label{eq:KevinMoore}
& \ell_1^-=0=\ell_j^-,&
&\ell_1^+
=
\left[\alpha_{1,j} \frac{\delta_{1,j}}{v_{1,j}}\right]^{\frac{\delta_{1,j}}{v_{1,j}^2 - \delta_{1,j}}},
&\ell_j^+ = \left[\alpha_{1,j} \frac{\delta_{1,j}}{v_{1,j}}\right]^{\frac{v_{1,j}^2}{v_{1,j}^2 - \delta_{1,j}}}, & & j \in \mathsf{J},
\end{align}
and do not satisfy \eqref{eq:Beardfish}.
\end{proposition}

\begin{proof} 
Fix $j \in \mathsf{J}$.
Since $c_{1,j} > 0$, by \eqref{eq:c1jlog} the formulas in \eqref{eq:kappaj} and \eqref{eq:ghlog} become
\begin{gather*}
c_{1,j}=v_{1,j} \, \frac{\ln(\ell_1^+)}{\ln(\ell_j^+)},\qquad
A_{1,j}=\alpha_{1,j} \, v_{1,j} \, \frac{\ln(\ell_1^+)}{\ln(\ell_j^+)},\qquad
k_j=0=\kappa_j,\\
g_h(\ell) = v_h \, \ell \, \ln\left(\frac{\ell_h^+}{\ell}\right),\qquad
g_h(0)=0.
\end{gather*}
Hence \eqref{eq:fjpmlog} can be written as 
\begin{equation}
\ell_j^+ \ln(\ell_j^+) = \alpha_{1,j} \, v_{1,j} \, \ell_1^+ \ln(\ell_1^+)
\label{eq:Opeth}
\end{equation}
and therefore \eqref{eq:maratonLOG} becomes
\begin{align*}
\left[ \delta_{1,j} - \frac{v_{1,j} \, \ell_j^+}{\alpha_{1,j} \, \ell_1^+} \right] \ln\left( \frac{\ell_j^+}{\ell} \right) = 0,&
&\ell \in (0, \ell_j^+),
\end{align*}
namely
\begin{equation}\label{eq:syst}
\ell_j^+ = \alpha_{1,j} \, \frac{\delta_{1,j}}{v_{1,j}} \, \ell_1^+.
\end{equation}
System \eqref{eq:Opeth}-\eqref{eq:syst} admits a solution if and only if either \eqref{eq:AmonAmarth} or \eqref{eq:mnk-1} holds true. In  the former case, \eqref{eq:Opeth}-\eqref{eq:syst} has infinitely many solutions and they satisfy \eqref{eq:Beardfish}; in the latter, the unique solution of \eqref{eq:Opeth}-\eqref{eq:syst} is \eqref{eq:KevinMoore}$_{2,3}$. We examine separately these cases.

Assume \eqref{eq:AmonAmarth}. In this case condition ($\mathcal{T}_l$) of Proposition \ref{p:TheMarsVoltaLOG} with $\ell_1^-=0=\ell_j^-$ is equivalent to $\ell_1^+=\ell_j^+ \in (0,1), \, j \in \mathsf{J}$, and then there are infinitely many traveling waves. They all satisfy \eqref{e:dens_cont} by Remark \ref{rem:AmonAmarth}.

Assume \eqref{eq:mnk-1}. In this case condition ($\mathcal{T}_l$) of Proposition \ref{p:TheMarsVoltaLOG} with $\ell_1^-=0=\ell_j^-$ is equivalent to $\ell_h^+ \in (0,1)$, $h \in \mathsf{H}$, satisfying \eqref{eq:Opeth}-\eqref{eq:syst}, namely to \eqref{eq:mnk-1}-\eqref{eq:KevinMoore}.
In particular, \eqref{eq:mnk-1}$_1$, \eqref{eq:KevinMoore} imply that $\ell_j^+$ and $\ell_1^+$ are distinct, namely they do not satisfy \eqref{eq:Beardfish}.
Moreover, by Remark \ref{rem:AmonAmarth} the traveling wave does not satisfies \eqref{e:dens_cont}.

At last, the reverse implications are direct consequences of previous discussion about the solutions of \eqref{eq:Opeth}-\eqref{eq:syst} and then the proof is complete.
\end{proof}

\appendix
\section{Proof of Theorem~\ref{t:E}}\label{s:A}

Let $\ell^\pm_h \in [0, 1]$ with $\ell^{-}_h\ne \ell^{+}_h $.  We introduce the change of variable
\begin{equation}\label{eq:change-1}
r_h \doteq \frac{\ell^{+}_h-\rho_h}{\ell^{+}_h-\ell^{-}_h},
\end{equation}
which implies $\rho_h = \ell^{+}_h - (\ell^{+}_h-\ell^{-}_h) \, r_h$,
$\rho_{h,t} = -(\ell^{+}_h - \ell^{-}_h) \, r_{h,t}$ and
$\rho_{h,x} = -(\ell^{+}_h - \ell^{-}_h) \, r_{h,x}$.
Consequently, equation \eqref{eq:model} can be written
\begin{equation}\label{eq:r-h}
r_{h,t} + G_h(r_h)_{x} = \left(E_{h}(r_h) \, r_{h,x}\right)_x,
\end{equation}
where
\begin{align*}
&G_{h}(r_h) \doteq -\frac{f_{h}\bigl(\ell^{+}_h - (\ell^{+}_h-\ell^{-}_h) \, r_h\bigr)-f_{h}(\ell^{+}_h)}{\ell^{+}_h -\ell^{-}_h},&
&E_{h}(r_h)\doteq D_{h}\bigl(\ell^{+}_h - (\ell^{+}_h-\ell^{-}_h) \, r_h\bigr).
\end{align*}
Furthermore,  equation \eqref{eq:r-h} has a wavefront solution $\psi_h$ from $1$ to $0$ with wave speed $\theta_h$  if and only if equation \eqref{eq:model} has a wavefront solution $\phi_h$ from $\ell^{-}_h$ to $\ell^{+}_h$ with the same speed. Notice that $\psi_h$ satisfies the equation
\begin{equation*}
\left(E_{h}(\psi_h)\psi'_{h}\right)' +\left(\theta_h - G'_{h}(\psi_h)\right) \psi'_h=0
\end{equation*}
and  $\phi_h$ is obtained by $\psi_h$ by the change of variable \eqref{eq:change-1}, i.e.
\begin{align}\label{eq:phi-1}
&\phi_h(\xi) = (\ell^{-}_h -\ell^{+}_h) \, \psi_h(\xi) + \ell^{+}_h,&
&\xi \in \mathbb{R}.
\end{align}

We discuss now  the existence of a wavefront solution $r_h(t,x)=\psi_h(x-\theta_h t+\sigma_h)=\psi_h(\xi)$ of \eqref{eq:r-h}.
In order to make use of \cite[Theorem~9.1]{GK}, we only need to show that
\begin{align}\label{eq:ghnb1}
&-G_{h}(r_{h})>-r_{h} \, G_{h}(1),&
&r_h \in (0,1).
\end{align}
By the definition of $G_h$ we have
\[
-r_h \, G_{h}(1) = -r_h \, \frac{f_{h}(\ell^{+}_h)-f_{h}(\ell^{-}_h)}
{\ell^{+}_h-\ell^{-}_{h}}.
\]
Then, inequality \eqref{eq:ghnb1} is equivalent to
\begin{align*}
&f_{h}(\ell^{+}_h) - \left(f_{h}(\ell^{+}_h)-f(\ell^{-}_h)\right) r_h < f_{h}\bigl(\ell^{+}_h - (\ell^{+}_h-\ell^{-}_h) \, r_h\bigr),&
&\hbox{ for }r_h \in (0,1),
\end{align*}
if and only if $\ell_h^-<\ell_h^+$. By the strict concavity of $f_h$ the last inequality is satisfied and then, by
 \cite[Theorem~9.1]{GK}, we deduce the existence of wavefront solutions $\psi_h$  from $1$ to $0$ for \eqref{eq:r-h}. The wave speed, in this case, is  $\theta_h \doteq G_h(1)$.
Furthermore, the profile $\psi_h$ is unique up to shifts and, if $\psi_h(0)\doteq\nu$ for some $0<\nu< 1$, then
\begin{equation}\label{eq:psi-1}
\begin{cases}
\psi_h(\xi)=1 & \hbox{for }\xi\le \nu_h^-,
\\
\ds\int_{\psi_{h}(\xi)}^{\nu} \frac{E_h(s)}{-G_h(s)+s \, G_h(1)}=\xi & \hbox{for } \nu_h^-<\xi<\nu_h^+ ,\\
\psi_h(\xi)=0 & \hbox{for }\xi\ge \nu_h^+,
\end{cases}
\end{equation}
where
\begin{align*}
&\nu_h^+ \doteq \int_{0}^{\nu} \frac{E_h(s)}{-G_h(s)+ s \, G_h(1)} \, {\rm d} s,&
&\nu_h^- \doteq - \int_{\nu}^{1} \frac{E_h(s)}{-G_h(s)+ s \, G_h(1)} \, {\rm d}s.
\end{align*}
Notice that, by differentiating \eqref{eq:psi-1} in the interval $(\nu_h^-, \nu_h^+)$, we obtain that
\begin{align}\label{e:derivpsi}
&\frac{E_h\left(\psi_h(\xi)\right)}{G_h\left(\psi_h(\xi)\right)-\psi_h(\xi) \, G_h(1)} \, \psi_h'(\xi)=1,&
&\xi \in (\nu_h^-, \nu_h^+),
\end{align}
which implies $\psi_h'<0$ in $(\nu_h^-, \nu_h^+)$ because of \eqref{eq:ghnb1}.

\noindent Consider now $\phi_h$ defined in \eqref{eq:phi-1}; it satisfies \eqref{eq:phi} with $I_h=(\nu_h^-, \nu_h^+)$ and $\phi_h'>0$ in $I_h$. Also condition \eqref{eq:ch} is true and $\phi_h \in \C2(I_h, (\ell^{-}_h, \ell^{+}_h)$ by the regularity of $D_h$ and $f_h$.

\noindent Now it remains to consider the boundary conditions of $\phi_h'$ at the extrema of $I_h$ in the different cases. We have the following.

\begin{itemize}
\item[{\em (i)}] Assume $\ell_h^-=0=D_h(0)$. We show that
\begin{equation}\label{eq:bbbm12}
\nu_h^-= - \int_{\nu}^{1} \frac{E_{h}(s)}{-G_{h}(s)+ s \, G_{h}(1)}\, {\rm d}s >-\infty.
\end{equation}
To prove \eqref{eq:bbbm12}, notice that $E_h(1)=D_h(0)=0$ and that $-G_{h}(s)+ sG_{h}(1) \to 0$ as $s\to 1^-$. In addition, by means of the strict concavity of $f_h$ we obtain that
\[
\lim_{s\to1^-} \frac{E'_{h}(s)}{-G_{h}'(s)+ G_{h}(1)}=\frac{E'_{h}(1)}{-G'_{h}(1)+ G_{h}(1)}=\frac{-\ell_h^+ \, D_h'(0)}{-f_h'(0)+
\frac{f_h(\ell_h^+)}{\ell_h^+}}\ge 0
\]
 and then, by applying de l'Hospital Theorem we prove condition \eqref{eq:bbbm12}. Moreover, by condition \eqref{e:derivpsi}, we get
 \begin{equation*}
\lim_{\xi\downarrow\nu_h^-}\psi_h'(\xi)=\begin{cases}
\ds-\frac{f_h'(0)-\frac{f_h(\ell_h^+)}{\ell_h^+}}{\ell_h^+ \, D_h'(0)}& \hbox{if } D_h'(0)>0,\\
-\infty & \hbox{if } D_h'(0)=0.
\end{cases}
\end{equation*}
By applying \eqref{eq:phi-1} we conclude that $\phi_h(\xi)=\ell_h^-$ for $\xi \le \nu_h^-$ and the estimates in \eqref{e:slope0} are satisfied.
Furthermore, by the change of variables \eqref{eq:change-1}, we obtain that
\begin{equation*}
\lim_{\xi\downarrow\nu_h^-}D_h(\phi_h(\xi))\phi_h^{\prime}(\xi)=\lim_{\xi\downarrow\nu_h^-}-\ell_h^+E_h(\psi_h(\xi))\psi_h^{\prime}(\xi)
\end{equation*}
and hence, by \eqref{e:derivpsi}, we deduce \eqref{eq:muh}.

\item[{\em (ii)}] Assume $1-\ell_h^+=0=D_h(1)$. With a similar reasoning as in \emph{(i)} we prove that $\nu_h^+>-\infty$. In fact, $E_h(0)=D_h(1)=0$ and $-s \, G_h(s)+s \, G_h(1) \to 0$ as $s \to 0^+$. Moreover
 \begin{equation*}
\lim_{s\to 0^+} \frac{E'_{h}(s)}{-G_{h}'(s)+ G_{h}(1)}=\frac{E'_{h}(0)}{-G'_{h}(0)+ G_{h}(1)}=\frac{\left( 1-\ell_h^- \right)\, D_h'(1)}{f_h'(1)+
\frac{f_h(\ell_h^-)}{1-\ell_h^-}}\ge 0
\end{equation*}
 and again, by applying de l'Hospital Theorem we prove that $\nu_h^+>-\infty$. Moreover, by the estimate \eqref{e:derivpsi}, we have that
\begin{equation*}
\lim_{\xi\uparrow\nu^h}\psi_h'(\xi)=\begin{cases}
\ds\frac{\frac{f_h(\ell_h^-)}{1-\ell_h^-}+f_h'(1)}{\left(\ell_h^- -1\right)\, D_h'(1)}& \hbox{if } D_h'(1)<0,\\
-\infty & \hbox{if } D_h'(1)=0.
\end{cases}
\end{equation*}
By applying \eqref{eq:phi-1} we conclude that $\phi_h(\xi)=\ell_h^+$ for $\xi \ge \nu_h^+$ and the estimates in \eqref{e:slope1} are satisfied; by \eqref{eq:change-1} and \eqref{e:derivpsi} we derive \eqref{eq:nuh}.

\item[{\em (iii)}] \ In all the other cases it is easy to show that $I_h=\mathbb{R}$ and again the slope condition \eqref{eq:TosinAbasi} can be obtained by the estimate \eqref{e:derivpsi}.
\end{itemize}

\section*{Acknowledgments}
The authors are members of the {\em Gruppo Nazionale per l'Analisi Matematica, la Probabilit\`a e le loro Applicazioni} (GNAMPA) of the {\em Istituto Nazionale di Alta Matematica} (INdAM). They were supported by the Project {\em Macroscopic models of traffic flows: qualitative analysis and implementation}, sponsored by the University of Modena and Reggio Emilia.
The first author was also supported by the project Balance Laws in the Modeling of Physical, Biological and Industrial Processes of GNAMPA.

{\small
\bibliography{refe_semifronts}
\bibliographystyle{abbrv}
}

\end{document}